\documentclass[11pt,a4paper, envcountsame]{amsart}
\usepackage[usenames,dvipsnames]{color}

 \usepackage[colorlinks,citecolor=blue,urlcolor=blue, linkcolor=blue, backref]{hyperref}
 
 \usepackage{amsmath, amssymb, xspace}
 \usepackage{graphicx}

 \usepackage[all]{xy}

\usepackage{tikz}
\usetikzlibrary{arrows,snakes,positioning,backgrounds,shadows}

\usepackage{undertilde}

\newtheorem{theorem}{Theorem}[section]

\newtheorem{conjtheorem}{Theorem?}[section]

\newtheorem{thm}[theorem]{Theorem}

\newtheorem{fact}[theorem]{Fact}

\newtheorem{proposition}[theorem]{Proposition}

\newtheorem{remark}[theorem]{Remark}

\newtheorem{prop}[theorem]{Proposition}

\newtheorem{claim}[theorem]{Claim}

\newtheorem{conjecture}[theorem]{Conjecture}

\newtheorem{lemma}[theorem]{Lemma}

\newtheorem{corollary}[theorem]{Corollary}

\newtheorem{cor}[theorem]{Corollary}

\newtheorem{question}[theorem]{Question}

\theoremstyle{definition}
\newtheorem{definition}[theorem]{Definition}

\newtheorem{df}[theorem]{Definition}

\newcommand{\CK}{\omega_1^{\mathrm{CK}}}
\newcommand{\KO}{\mathcal{O}}
\newcommand{\DII}{\Delta^0_2}
\newcommand{\NN}{{\mathbb{N}}}
\newcommand{\RR}{{\mathbb{R}}}
\newcommand{\R}{{\mathbb{R}}}
\newcommand{\QQ}{{\mathbb{Q}}}

\newcommand{\sub}{\subseteq}
\newcommand{\sN}[1]{_{#1\in \NN}}
\newcommand{\uhr}[1]{\! \upharpoonright_{#1}}
\newcommand{\ML}{Martin-L{\"o}f}
\newcommand{\SI}[1]{\Sigma^0_{#1}}
\newcommand{\PI}[1]{\Pi^0_{#1}}

\newcommand{\bi}{\begin{itemize}}
\newcommand{\ei}{\end{itemize}}
\newcommand{\bc}{\begin{center}}
\newcommand{\ec}{\end{center}}

\newcommand{\Halt}{{\ES'}}
\newcommand{\ES}{\emptyset}
\newcommand{\estring}{\emptyset}

\newcommand{\tp}[1]{2^{#1}}
\newcommand{\ex}{\exists}
\newcommand{\fa}{\forall}

\newcommand{\la}{\langle}
\newcommand{\ra}{\rangle}
\newcommand{\Kuc}{Ku{\v c}era}
\newcommand{\seqcantor}{2^{ \NN}}
\newcommand{\cs}{2^\omega}
\newcommand{\cantor}{\seqcantor}
\newcommand{\strcantor}{2^{ < \omega}}
\newcommand{\fs}{2^{<\omega}}
\newcommand{\strbaire}{\omega^{ < \omega}}

\newcommand{\Opcl}[1]{[#1]^\prec}
\newcommand{\leT}{\le_{\mathrm{T}}}

\newcommand{\MLR}{\mbox{\rm \textsf{MLR}}}
\newcommand{\Om}{\Omega}
\newcommand{\n}{\noindent}

\newcommand{\vsps}{\vspace{3pt}}
\newcommand{\verif}{\n {\it Verification.\ }}
\newcommand{\vsp}{\vspace{6pt}}
\newcommand{\leb}{\mathbf{\lambda}}
\newcommand{\lwtt}{\le_{\mathrm{wtt}}}

\newcommand{\sss}{\sigma}
\newcommand{\aaa}{\alpha}
\newcommand{\bbb}{\beta}
\newcommand{\w}{\omega}

\newcommand{\lland}{\, \land \, }

\newcommand \seq[1]{{\left\langle{#1}\right\rangle}}

\newcommand\+[1]{\mathcal{#1}}
\newcommand{\sC}{\+ C}

\newcommand{\ol}{\overline}
\newcommand{\ul}{\underline}
\newcommand{\ape}{\hat{\ }}

\newcommand{\LLR}{\Longleftrightarrow}
\newcommand{\lra}{\leftrightarrow}
\newcommand{\LR}{\Leftrightarrow}
\newcommand{\RA}{\Rightarrow}
\newcommand{\LA}{\Leftarrow}
\newcommand{\DA}{\downarrow}

\newcommand{\rapf}{\n $\RA:$\ }
\newcommand{\lapf}{\n $\LA:$\ }

\newcommand{\UM}{\mathbb{U}}

\newcommand{\range}{\ensuremath{\mathrm{range}}}
\newcommand{\dom}{\ensuremath{\mathrm{dom}}}

\def\uh{\upharpoonright}

\DeclareMathOperator \init{init}

\newcommand \andd{\,\,\,\&\,\,\,}

  \DeclareMathOperator{\High}{High}
  
  \newcommand{\Dem}{\mbox{\rm \textsf{Demuth}}}
  \newcommand{\SR}{\mbox{\rm \textsf{SR}}}
  
\begin{document}

\title{Logic Blog 2012}

 \author{Editor: Andr\'e Nies}

\maketitle

\begin{abstract}  The 2012  logic blog has focussed on the following:
Randomness and computable analysis/ergodic theory; Systematizing algorithmic randomness notions; 
Traceability;
Higher randomness;
Calibrating the complexity of equivalence relations from computability theory and algebra.
 \end{abstract}

%
%
%

\n The  \href{http://dx.doi.org/2292/9821}{Logic Blog 2010}. (Link: \texttt{http://dx.doi.org/2292/9821})  

\vsps

\n The    \href{http://dx.doi.org/2292/19205}{Logic Blog 2011}. (Link: \texttt{http://dx.doi.org/2292/19205})



%
%
%

\medskip

Postings can be cited.  An example of a citation is:

\medskip

\n  I.\  Kallimullin and A.\  Nies, \emph{Structures that are   computable almost surely}, Logic Blog, March 2010, available at 
\texttt{http://dx.doi.org/2292/9821}. 

%

\medskip

\tableofcontents


\part{Randomness and computable analysis}

\section{Computability of Ergodic Convergence (Towsner)}

Henry Towsner was at the Feb Oberwolfach meeting and discussed with  Nies and Bienvenu.

Let $\mu$ be a computable measure on $\Omega$, let $T$ be a computable, measure-preserving transformation.  Henry described  an example of an $L^1$ computable function where the   limit in the sense of Birkhoff is not $L^1$ computable. This was  in response to a question of Nies and Bienvenu.

\begin{definition} For any function $f$, we write $A_Nf$ for the function \[\frac{1}{N}\sum_{i<N}f(T^ix).\]
\end{definition}

\n Recall the following.
\bi \item If $f \in L^1(\mu) $ then by Birkhoff ergodic theorem,  \bc $\ol f(x):= \lim_N A_N f(x)$ \ec  exists a.e.\ and $\ol f$  is in $L^1(\mu)$. \item  If in fact $f \in L^2(\mu) $  then the $A_N f$ converge in the $L^2$ norm by the main ergodic theorem.
\ei

\begin{definition} {\rm We say the mean rate of convergence of $f$ is computable if the $A_N f$ converge computably, namely, there is a computable function $p$ such that for every $n$ and every $m\geq p(n)$, \[||A_{p(n)}f-A_mf||_{L^2}<1/n.\]
We say the pointwise rate of convergence of $f$ is computable if there is a computable function $p$ such that for every $n$ and every $m\geq p(n)$, \[\mu(\{x\mid \exists i\in[p(n),m]|A_{p(n)}f(x)-A_i(x)|>1/n\})<1/n.\] } \end{definition}

\begin{theorem} Let $\mu$ be a computable measure on $\Omega$, let $T$ be a computable, measure-preserving transformation, and let $f$ be $L^1$-computable with respect to $\mu$ such that $||f||_{L^2}$  exists (i.e., $f \in L^2(\mu)$).
Then the following are equivalent: 
\begin{enumerate} \item[(1)] $||\overline{f}||_{L^2}$ is computable, 
	\item[(2)] The mean rate of convergence of $f$ is computable,

 \item[(3)] The pointwise rate of convergence of $f$ is computable,

  \item[(4)] $\overline{f}$ is $L^1$-computable.  \end{enumerate} \end{theorem}
	
	 \begin{proof} For the first three, the implications (1)$\Rightarrow$(2) and (1)$\Rightarrow$(3) are shown in \cite{Avigad.Gerhardy.ea:08} while the implications (3)$\Rightarrow$(2) and (2)$\Rightarrow$(1) are trivial.
For the equivalence with (4), it is obvious that (3) implies (4). For the converse, not that $||\overline{f}||_{L^2}\leq||f||_{L^2}$, and in particular $||\overline{f}||_{L^2}$ is bounded. Therefore (4) implies (2). \end{proof}
I don't know of any proof that (4) implies (3) in the absence of a bound on the $L^2$ norm. This is almost certainly still true, but a bit more work is required, and there are some techicalities in the proof which might create problems. (Specifically, it's not clear to me that every $L^1$ computable function is computably approximable by $L^2$ functions.)

An example where all four properties fail is given in \cite{Avigad.Gerhardy.ea:08}.
Divide $\Omega$ into countable many components. On the $i$-th component, $T$ is supposed to represent a ``rotation'' of $2^{-j}$ if the $i$-th Turing machine halts in exactly $j$ steps, and the identity if the the $i$-th Turing machine never
halts. (By a rotation, I mean adding $2^{-j}$ to each element in the component, wrapping around if we overflow out of that component.)
More precisely, consider some sequence $1^i0\sigma=\tau\in\Omega$. (That is, the sequence starts with an initial segment of exactly $i$ $1$'s.) If the $i$-th Turing machine halts at step $j$, \[T\tau=\left({2^{i+1}\sigma+2^{-j}\mod 1}\right)2^{-(i+1)}.\]
If the $i$-th Turing machine never halts, $T\tau=\tau$.
Consider the set $A=\bigcup_i [1^i01]$. On any element $\tau\in[1^i00]$, $\overline{\chi_A}(\tau)$ is $0$ if the $i$-th Turing machine never halts and $1/2$ if the $i$-th Turing machine does halt. In particular, given a computable $f$ such that $||f-\overline{\chi_A}||_{L^1}<2^{-(i+4)}$, we may computably find subintervals of $[1^i00]$ of collective measure $2^{-(i+3)}$ such that $f$ is defined on these subintervals, and the $i$-th Turing machine halts iff the average of $f$ over these subintervals is $>1/4$.

\section{Lebesgue Density and Lebesgue differentiation}
\label{s:LDT_and_LDiffT} \label{ss:LDT} 

The following is from  the preprint \cite{Bienvenu.Greenberg.ea:preprint}; a shortened version has been  submitted, but does   not contain the results below.

For measurable sets $P,A\subseteq \R$ with $A$ non-null, $\leb(P|A) = \leb(P\cap A)/\leb(A)$ is the conditional measure (probability) of $P$ given $A$.  Recall that the lower density of a measurable set $P\subseteq \R$ at a point $z\in \R$ is
\[ \rho (P|z) = \liminf_{h\to 0} \left\{  \leb(P|I) \,:\, I \text{ is an open interval, }z\in I \andd |I|< h   \right\} .\]
 Intuitively, $\rho(P|z)$ gauges the fraction of space filled by~$P$ around~$z$ if we ``zoom in'' arbitrarily close to $z$. 

Lebesgue's density theorem \cite[page 407]{Lebesgue:1910} says that for any measurable set $P$, for almost all $z\in P$ we have $\rho(P|z)=1$. An effective version of this theorem is given by identifying a collection of effectively presented sets $P$ and the collection of random point $z$ for which $\rho(P|z)=1$ for all sets $P$ in the collection containing $z$ as an element. Since the theorem is immediate for open sets, the simplest nontrivial effective version is obtained by choosing $P$ to range over the collection of effectively closed subsets of $\R$. We call a real number $z\in \R$ a \emph{density-one point} if for every effectively closed set $P$ containing $z$ we have $\rho(P|z)=1$.

\begin{definition} \label{def:intervalce}  
A non-decreasing, lower semicontinuous function $f\colon [0,1]\to \R$ is \emph{interval-c.e.}\ if $f(0)=0$, and $f(y)-f(x)$ is a left-c.e.\ real, uniformly in rationals $x<y$. 
\end{definition}

We investigate density-one points regardless of randomness. Our investigations yield more information about interval-c.e.\ functions. 

Recall that a measurable function $g\colon [0,1]\to \R$ is \emph{integrable} if $\int_{[0,1]} |g|\,\textup{d}\leb$ is finite. 

A  real $z \in [0,1]$ is called a \emph{Lebesgue point} of an integrable function $g$ if 
 \[ \lim_{  |Q| \to 0} \frac{1}{|Q|} \int_Q g\,\textup{d}\leb = g(z),  \tag{$*$} \]
 where $Q$ ranges over open intervals containing $z$. A real $z$ is called a \emph{weak Lebesgue point} if the limit in $(*)$ exists (but may be different from $g(z)$). 

The Lebesgue differentiation theorem states that for any integrable function $g$, almost every point $z \in [0,1]$ is a Lebesgue point of $g$.

If $g$ is the characteristic function $1_\sC $ of a measurable set $\sC$,  then the limit $(*)$ above equals the density $\rho(\sC| z)$. Thus, one can view the density theorem as a special case of  the differentiation theorem.


As with the density theorem, an effective version of the Lebesgue differentiation theorem is obtained by specifying a collection $\mathfrak F$ of effectively presented integrable functions. For example, Pathak, Rojas and Simpson~\cite{Pathak.Rojas.ea:12} and independently Freer et al.\ \cite{Freer.Kjos.ea:nd} studied the Lebesgue differentiation theorem for $L^1$-computable functions; Freer et al.\ also considered the collection of  $L^p$-computable functions for a computable real $p \ge 1$. They showed that  Schnorr randomness of $z$ is equivalent to being a weak Lebesgue point for each such function; the implication left to right is due to Pathak et al. Here we consider bounded lower semicomputable functions; we will see that every Oberwolfach random point is a Lebesgue point of any bounded lower semicomputable function, and a weak Lebesgue point of any integrable lower semicomputable function. Indeed, we observe that the effective versions of the Lebesgue density and differentiation theorems that we consider in this paper are equivalent for \emph{any} real number $z$.

\begin{prop}\label{prop:DensityLebesgueDiff}
	The following are equivalent for a real $z\in [0,1]$.	
	\begin{itemize}
	\item[(i)] $z$ is a density-one point.
	
	\item[(ii)] $z$ is a Lebesgue point of every bounded upper semi-computable function  $g\colon [0,1] \rightarrow \R$.
	
	\item[(iii)] $z$ is a Lebesgue point of every bounded lower semi-computable function  $g\colon [0,1] \rightarrow \R$.
	\end{itemize}
\end{prop}

Since $[0,1]$ is compact, every lower semi-continuous function on $[0,1]$ is bounded from below, and every upper semi-continuous function on $[0,1]$ is bounded from above. Hence in (ii) we could merely require that the upper semi-computable function $g$ be bounded from below, and in (iii), that the function be bounded from above.

\begin{proof} \
\noindent (ii) $ \Rightarrow$  (i)  is immediate. Indeed, if $\sC$ is effectively closed then the density of $\sC$ at $z$ is precisely the limit in ($*$) for $g = 1_\sC$, the characteristic function of $\sC$. The function $1_\sC$ is upper semi-computable and is integrable.

\smallskip

\noindent  (iii) $\Rightarrow$ (ii).   If $g$ is upper semi-computable and integrable then $-g$ is lower semi-computable and integrable. 

\smallskip

\noindent  (i) $\Rightarrow$ (iii).  Let $z$ be a density-one point. We show, in three steps, that $z$ is a Lebesgue point of every integrable lower semi-computable function. 

First, let $g = 1_\sC$ for an effectively closed set $\sC$. If $z\in \sC$ then the equality ($*$) holds at $z$ because $z$ is a density-one point. If $z \not \in \sC$ then $z$ is a Lebesgue point of $g$ because the complement of $\sC$ is open. 

Second, the property of being a Lebesgue point is preserved under taking linear combinations of functions. We conclude that $z$ is a Lebesgue point for all linear combinations of characteristic functions of effectively closed sets. 

Finally, let $g$ be any bounded lower semi-computable function. By scaling and shifting, we may assume that $g$ is bounded between $0$ and $1$.

We approximate $g$ by a step-function. For $x\in [0,1]$, let $f(x)$ be the greatest integer multiple of $2^{-k}$ which is bounded by $g(x)$. For all $i\le 2^k$, let 
\[ \sC_i = g^{-1}(-\infty, i\cdot 2^{-k}].\]
Each set $\sC_i$ is effectively closed, and $f = 2^{-k} \sum_{i=1}^{2^k} 1_{\sC_i}$. Then $||f-g||_\infty \le 2^{-k}$, which implies that for any interval $Q$, 
\[ \left|  \frac{1}{|Q|} \int_{Q} g\, \textup{d}\leb - \frac{1}{|Q|} \int_{Q} f \, \textup{d}\leb \right| \le 2^{-k} .\]

Because $f$ is a linear combination of characteristic functions of effectively closed sets, we know that for sufficiently short intervals $Q$ containing $z$ we have 
\[ \left| f(z) - \frac{1}{|Q|}\int_Q f\,\textup{d}\leb \right| < 2^{-k} .\]
Because $|f(z)-g(z)| \le 2^{-k}$ we conclude that for sufficiently short intervals $Q$ containing $z$, we have 
\[ \left| g(z) - \frac{1}{|Q|}\int_Q g\,\textup{d}\leb \right| < 3\cdot 2^{-k} .  \qedhere\]
\end{proof}

\begin{question} \label{question:integrable_lscomp_density-one}
	If $z$ is a density-one point, is $z$ a Lebesgue point of every \emph{integrable} lower semicomputable function? 
\end{question}

Recall  that a non-decreasing, lower semicontinuous function $f\colon [0,1]\to \R$ with $f(0)=0$ corresponds to a measure $\mu_f$ on $[0,1)$, determined by $\mu([x,y)) = f(y)-f(x)$. The measure $\mu_f$ is absolutely continuous with respect to Lebesgue measure if and only if the function $f$ is an absolutely continuous function. In this case, the Radon-Nikodym theorem says that $\mu_f(A) = \int_A g\,\textup{d}\leb$ for some non-negative integrable function $g$. A real $z$ is a Lebesgue point of $g$ if and only if $f'(z)$ exists and equals $g(z)$, and a weak Lebesgue point if and only if $f'(z)$ exists. 

If $g$ is lower semicomputable then $f$ is interval-c.e.\  

\begin{corollary}\label{cor:OW_weak_Lebesgue_point}
	Every Oberwolfach random real is a weak Lebesgue point of every integrable lower semicomputable function. 
\end{corollary}

\begin{proof}
	If $g$ is lower semi-computable, then it is bounded from below, and so by adding a constant we may assume it is positive. Then apply Theorem~\ref{thm:OW_random_differentiable}. 
\end{proof}

A weaker version of Question~\ref{question:integrable_lscomp_density-one} is:

\begin{question}
	Is every Oberwolfach random real a Lebesgue point of every integrable lower semicomputable function?
\end{question}

\

Finally, we see that the relation between non-negative, integrable lower semicomputable functions and absolutely continuous interval-c.e.\ functions is not a correspondence. The next result  shows that there is an interval-c.e.\ function $f$ which is not the distribution function $\int_0^x g\,\textup{d}\leb$ for any lower semicomputable function $g$, indeed not for any lower semicontinuous function $g$.

\begin{proposition}\label{pro:interval_c.e._AC_not_integral} There is nondecreasing  computable (hence, interval-c.e.) Lip\-schitz function $f$ that is not of the form $f(x) = \int_0^x g\,\textup{d}\leb$ for any lower semicontinuous function $g$.
\end{proposition}
 \begin{proof} Let $M$ be a computable martingale that succeeds on any   $Z \in \cantor$ failing the law of large numbers. By Theorem 4.2 of \cite{Freer.Kjos.ea:nd} (and its proof) there is a computable Lipschitz function $f$ such that $f'(z)$ fails to exist whenever $M$ succeeds  on a binary expansion $Z$ of $z$. Adding a linear term, we may assume that  $f$ is nondecreasing. Now suppose $f(x) = \int_0^x g\,\textup{d}\leb$ for a lower semicontinuous function~$g$. If $C$ is a Lipschitz constant for $f$, then $\{x\colon \, g(x) > C\}$ is a null set. Since this set is also open, it is  empty. Hence $g$ is bounded. 
	
	 If $Z$ is $1$-generic relative to a representation of $g$ then $Z$ is a density-one point relative to this representation of $g$ by relativizing  the observation in \cite{Bienvenu.Hoelzl.ea:12a} mentioned earlier on. Hence by   Proposition~\ref{prop:DensityLebesgueDiff} in relativized form,  $z$ is a Lebesgue point of $g$. Then $f'(z)$ exists. 
	
	On the other hand, each 1-generic $Z$ fails the law of large numbers. So $M$ succeeds on $Z$, and $f'(z)$ does not exist.  Contradiction. \end{proof}

	\section{Randomness notions  and  their corresponding lowness classes}
\subsection{Summary of randomness notions} The following diagram gives an overview of the randomness notions we have discussed. They are  stronger than, but close to, ML-randomness. The diagram is a modification of a similar diagram in \cite{Bienvenu.Hoelzl.ea:12a}.	We consider properties of a given  ML-random real. 

\vsp

\begin{center}
\newcommand{\specialcell}[2][c]{%
  \begin{tabular}[#1]{@{}c@{}}#2\end{tabular}}

\begin{tabular}{rcccccccl}
\specialcell{\emph{Oberwolfach} \\ \emph{random}}  & \hspace*{-0.1cm}$\xrightarrow{}$ &  \specialcell{ all  interval-c.e. \\ functions are  \\ differentiable at \\ the real} &
	\hspace*{-0.1cm}$\xrightarrow{}$\hspace*{-0.1cm} &
\specialcell{density-one   \\  point} &
	\hspace*{-0.2cm}$\xrightarrow{\hspace*{0.5cm}}$\hspace*{-0.2cm} &
\specialcell{positive\\density point} \\
  \rule{0cm}{16px}$\Big\uparrow$\hspace*{.4cm}  & & 	& & & & \rule{0cm}{16px}$\Big\updownarrow$\hspace*{0cm} \\
 not LR-hard & & 	& \hspace*{-0.1cm}$\xrightarrow{}$  & & & \hspace*{-0.4cm}  \specialcell{Turing incomplete} & &
\end{tabular}
\end{center}

\vsp

The rightmost vertical double arrow refers to a result of Bienvenu et al.\  \cite{Bienvenu.Hoelzl.ea:12a}.   Random pseudo-jump inversion implies that the implication not LR-hard $\to $ Turing incomplete is proper.  In fact, by \cite{Day.Miller:nd}, the implication density-one point $\to$ positive density point is proper. 

One way to separate these notions when  viewed as operators on oracles would be to separate the corresponding lowness classes. Recall that an oracle $A$ is low for a randomness notion $\sC$ if $\sC^A= \sC$. More generally, $A$ is low for a pair of randomness notions $\sC \sub \+ D$  if $\sC \sub \+ D^A$.  Combining results in \cite{Downey.Nies.ea:06,Nies:AM} shows that for the pair $\sC= $ weak 2-randomness and $\+ D=$ ML-randomness, the double lowness class coincides with $K$-triviality. Thus, the lowness class for any of the notions above is contained in the $K$-trivials. 

Using the recent result of Day and Miller \cite{Day.Miller:12}, Franklin and Ng~\cite{Franklin.Ng:10} have shown that lowness for difference randomness coincides with $K$-triviality. We now obtain such a coincidence for two further notions in the diagram above: density-one points, and being not LR-hard.

\begin{prop} Let $A$ be $K$-trivial. 
	
\bi \item[(1)] $A$ is low for the notion ``density-one $\cap$ ML-random''.  

\item[(2)] $A$ is low for the notion ``non-LR-hard $\cap$ ML-random''. \ei \end{prop}

\begin{proof} Each time we need to show that if $Z$ is not $A$-random in the given sense, then $Z$ is not random in that sense. Since $A$ is low for ML-randomness and our notions imply ML-randomness, we may assume that $Z$ is ML-random.
	
	\vsp

	\n (1).  Suppose    $Z \in \+ P$ for some $\PI 1 (A)$ class $\+P$  with $\rho(Z \mid P) <1 $.  \\
   Since  $A$ is  $K$-trivial and $Z$ random, by the argument of Day and Miller \cite{Day.Miller:12} there is a $\PI 1$  class $\+ Q$ with $\+ P \supseteq \+Q \ni Z$.   
 Then $\rho(Z\mid \+ Q)< 1$.   
 
\vsp

\n (2).  Let $\MLR$ denote the class of ML-randoms. We modify an argument of Hirschfeldt \cite[8.5.15]{Nies:book}. Suppose that $Z$ is LR-hard relative to $A$, namely,  $Z \oplus A \ge_{LR} A'\equiv_T \Halt$. We show  $Z \ge_{LR} \Halt$. Suppose that  $Y \in \MLR^Z$. Then $ Y\oplus Z \in \MLR$ by van Lambalgen's theorem, so $ Y\oplus Z \in \MLR^A$. This implies $Y \in \MLR^{Z\oplus A}$ by van Lambalgen's theorem relative to $A$. By our hypothesis on $Z$, this implies that   $Y$ is ML-random relative to $\Halt$.  Thus, $Z \ge_{LR} \Halt$. 
\end{proof}

%

\newpage

\part{Randomness, Kolmogorov complexity,  and computability}

\section{Cupping $\DII$ DNR sets (Joseph S.\ Miller, June 2012)}


Nies, Stephan and Terwijn \cite{Nies.Stephan.ea:05} proved that if $Z\leq_T\emptyset'$ is Martin-L\"of random and $C\in 2^\omega$ is a c.e.\ set, then either $Z\oplus C\geq_T \emptyset'$ or $Z$ is Martin-L\"of random relative to $C$. We prove an analogous result with ``is Martin-L\"of random'' replaced by ``has DNC degree''.

\begin{thm}\label{thm:DNC-join-ce}
Assume that $X\in 2^\omega$ has DNC degree and $C\in 2^\omega$ is a c.e.\ set. Either $X\oplus C\geq_T \emptyset'$ or $X$ has DNC degree relative to $C$.
\end{thm}
\begin{proof}
Let $g\colon\omega\to\omega$ be an $X$-computable DNC function. By the recursion theorem, we may assume that we control an infinite sequence of positions $\{k_{n,m}\}_{n,m\in\omega}$ of the diagonal function $e\mapsto\phi_e(e)$. If $n$ enters $\emptyset'$ at stage $s$, then define
\[
\phi_{k_{n,m}}(k_{n,m}) = \phi^{C_s}_{k_{n,m},s}(k_{n,m}),
\]
if the latter converges. If there is an $n$ such that $m\mapsto g(k_{n,m})$ is a DNC function relative to $C$, then we are done. If not, define $f\leq_T X\oplus C$ such that $f(n)$ is the least $s$ such that $(\exists m\leq s)\; g(k_{n,m})=\phi^{C_s}_{k_{n,m},s}(k_{n,m})$ for a $C$-correct computation. By assumption, $f$ is total. If $n$ enters $\emptyset'$, then it must happen at a stage $s<f(n)$. Otherwise, we would contradict the fact that $g$ is DNC. Therefore, $\emptyset'\leq_T f\leq_T X\oplus C$.
\end{proof}

This result is quite similar to, and was motivated by, a beautiful theorem of Day and Reimann \cite[Corollary~8.2.1]{Day:11}. They proved that if $X\in 2^\omega$ has PA degree and $C\in 2^\omega$ is a c.e.\ set, then either $X\oplus C\geq_T \emptyset'$ or $X\geq_T C$. The conclusion can fairly easily be strengthened to highlight the similarity with Theorem~\ref{thm:DNC-join-ce} (Day, personal communication, October 2011).

\begin{thm}[Day and Reimann]
Assume that $X\in 2^\omega$ has PA degree and $C\in 2^\omega$ is a c.e.\ set. Either $X\oplus C\geq_T \emptyset'$ or $X$ has PA degree relative to $C$.
\end{thm}
\begin{proof}
Apply the result of Day and Reimann to $X$ and $C$. If $X\oplus C\geq_T \emptyset'$, we are done. Otherwise, $X\geq_T C$ and $X\ngeq_T\emptyset'$. Take $Y\in 2^\omega$ such that $Y$ has PA degree and $X$ has PA degree relative to $Y$, which is possible by Simpson~\cite[Theorem~6.5]{Simpson:77}. Note that $X\geq_T Y$, so $Y\oplus C\leq_T X\ngeq_T\emptyset'$. Applying the result of Day and Reimann to $Y$ and $C$ gives $Y\geq_T C$. Therefore, $X$ has PA degree relative to $C$.
\end{proof}

The proof of Theorem~\ref{thm:DNC-join-ce} could be used, with only superficial modification, to prove this result. Ku\v cera (2011) also gave a direct proof of Day and Reimann's result.

Note that the DNC version of the original Day and Reimann result is false. In other words, we cannot replace ``$X$ has DNC degree relative to $C$'' with ``$X\geq_T C$'' in Theorem~\ref{thm:DNC-join-ce}. To see this, let $C\in 2^\omega$ be a low c.e.\ set that is not $K$-trivial. By the low basis theorem relative to $C$, there is a Martin-L\"of random $Z\in 2^\omega$ such that $Z\oplus C$ is low. So $Z$ has DNC degree and $Z\oplus C\ngeq_T \emptyset'$. But $Z\geq_T C$ would imply that $C$ is $K$-trivial by Hirschfeldt, Nies and Stephan \cite{Hirschfeldt.Nies.ea:07}.

We now consider another theorem of Nies, Stephan and Terwijn \cite{Nies.Stephan.ea:05}. They proved that if $Z\leq_T\emptyset'$ is Martin-L\"of random relative to $A$, then $A$ is GL$_1$ (i.e., $A'\leq A\oplus\emptyset'$). Any $Z$ that is Martin-L\"of random relative to $A$ has DNC degree relative $A$, so the following theorem generalizes their result.

\begin{thm}\label{thm:GL1}
Assume that $X\in 2^\omega$ is $\Delta^0_2$ and has DNC degree relative to $A\in 2^\omega$. Then $A$ is GL$_1$. 
\end{thm}
\begin{proof}
Let $g\colon\omega\to\omega$ be an $X$-computable DNC function relative to $A$. Because $g$ is $\Delta^0_2$, there is a computable $h\colon\omega^2\to\omega$ such that $(\forall n)\;g(n) = \lim_{s\to\infty} h(n,s)$. By the relativized recursion theorem, we may assume that we $A$-computably control an infinite \emph{computable} sequence of positions $\{k_n\}_{n\in\omega}$ of the diagonal function relative to $A$, i.e., $e\mapsto\phi^A_e(e)$. If $n$ enters $A'$ at stage $s$, then let $\phi^A_{k_n}(k_n) = h(k_n,s)$. Define $f\leq_T \emptyset'$ such that $f(n)$ is the least $s$ such that $(\forall t\geq s)\; g(k_n) = h(k_n,t)$. If $n$ enters $A'$, then it must happen at a stage $s<f(n)$. Otherwise, we would contradict the fact that $g$ is DNC relative to $A$. Therefore, $A'\leq_T A\oplus f\leq_T A\oplus \emptyset'$.
\end{proof}

Ku\v cera and Slaman (1989) built an incomplete c.e.\ set that cups every $\Delta^0_2$ DNC degree to $\emptyset'$. Bienvenu, Greenberg, Ku\v cera, Nies and Turetsky (2012) showed that, in fact, every superhigh c.e.\ set has this property. Their proof uses Kolmogorov complexity; Ku\v cera gave an alternate and purely computability-theoretic proof. We improve their result further by showing that any non-low c.e.\ set cups every $\Delta^0_2$ DNC degree to $\emptyset'$.

\begin{cor}
If $C\in 2^\omega$ is a non-low c.e.\ set and $X\in 2^\omega$ has $\Delta^0_2$ DNC degree, then $X\oplus C\equiv_T\emptyset'$.
\end{cor}
\begin{proof}
By Theorem~\ref{thm:DNC-join-ce}, either $X\oplus C\geq_T\emptyset'$ or $X$ has DNC degree relative to $C$. In the latter case, Theorem~\ref{thm:GL1} implies that $C$ is GL$_1$, hence low. This is not true, so $X\oplus C\geq_T\emptyset'$. Clearly, $X\oplus C\leq_T\emptyset'$.
\end{proof}

Note that if $C\in 2^\omega$ is a low set, then the low basis theorem relativized to $C$ gives us an $X$ that has DNC (even PA) degree such that $X\oplus C$ is low. Therefore, the corollary is tight: no low (c.e.) set cups every $\Delta^0_2$ DNC degree to $\emptyset'$.

\section{Cupping $\DII$ DNR sets (Bienvenu, \Kuc, et al., Feb.\ 2012)}
The following was obtained during the Research in Pairs stay at MFO, of  Bienvenu,  Greenberg,  \Kuc, Nies and Turetsky. It also provides a short proof of a 1989 result by \Kuc\  and Slaman who built a c.e.\ incomplete set that cups all $\DII$ DNR sets above $\ES'$. The idea to use Kolmogorov complexity~$K$ is due to Bienvenu.
There also is  a new  proof not using $K$ but much shorter than the original construction; this is  due to \Kuc . See Subsection~\ref{ss:not_K}.

\subsection{A proof using Kolmogorov complexity}
\begin{lemma}\label{lem:busy-beaver}
Let $B$ be the function $B(n)=\min \{t \in \NN \mid \forall s > t, \  K(s)\geq n\}$. Any function dominating~$B$ computes $\emptyset'$. 
\end{lemma}

\begin{theorem}\label{thm:kolmo-criterion}
  Let $A$ be a c.e.\ set such that $K^A(\sigma) \leq^+ f(K^{\emptyset'}(\sigma))$ for some $\DII$ function $f$. Then $A$ joins every $\Delta^0_2$ DNR set above $\emptyset'$.
\end{theorem}

Taking $f(x)=x+O(1)$ in the theorem, this shows that any LR-hard c.e. degree joins every $\Delta^0_2$ DNR set above $\emptyset'$, and in particular there exists an incomplete such c.e.\ set.

\begin{proof}
A first remark: the bigger the function $f$, the stronger the result, so we can assume that $f$ is increasing and $f(n)>4n$ for all~$n$. Let $A$ be such a set and $D$ be a set of DNR degree. We use a result of Kjos-Hanssen, Merkle and Stephan: $D$ having DNR degree is equivalent to $D$ computing a sequence $(\sigma_n)$ of strings such that $K(\sigma_n) \geq n$. $D$ being $\emptyset'$-computable, using $\emptyset'$, we can compute the sequence $\sigma_{2f(n)}$, thus $K^{\emptyset'}(\sigma_{2f(n)}) \leq^+ K^{\emptyset'}(n) \leq^+ 2 \log n$. By the assumption on~$A$, $K^A(\sigma_{f(2n)}) \leq^+ f(2 \log n) \leq^+ f(n)$. On the other hand $K(\sigma_{2f(n)}) \geq 2f(n)$. Informally this means that $A$ contains a lot of information about the $\sigma_{2f(n)}$ (it makes the Kolmogorov complexity of $\sigma_{2f(n)}$ drop from at least $2f(n)$ to at most $f(n)$). 

We now show how to use $A \oplus D$ to compute $\emptyset'$, using Lemma~\ref{lem:busy-beaver}. Given $n$, use $A \oplus D$ to do the following. First, using $D$, compute the sequence $\sigma_i$, and using $A$ look for an index $i$ such that $K^A(\sigma_i) \leq i-2n$. Such an $i$ exists as $K^A(\sigma_{2f(n)}) \leq^+ f(n) \leq^+ 2f(n)-2n$. Finding such an~$i$ means finding a program $p$ of length at most $i-2n$ for the $A$-universal machine $\UM^A$. Let $u$ be the use of $A$ in the computation $\UM(p)=\sigma_i$ and let $t_n$ be the settling time of $A \upharpoonright u$. We claim that for any sufficiently large~$n$, $t_n \geq B(n)$, which by Lemma~\ref{lem:busy-beaver} will prove the result. Let thus $s$ be any integer bigger that $t_n$. First, notice that
\begin{equation}
K(\sigma_i) \leq^+ |p|+K(s) 
\end{equation}

Indeed, if one knows $p$ and $s$, one can compute $U^{A_s}(p)=\UM^{A}(p)=\sigma_i$ (the first equality comes from the definition of $s$). Since $K(\sigma) \geq i$ and $|p| \leq n-i$, it follows that 
\begin{equation}
K(s) \geq^+ K(\sigma_i) - |p| \geq^+ i- (i-2n) \geq^+ 2n
\end{equation}

And thus for $n$ large enough $K(s) \geq n$.

\end{proof}
For the definition of JT-reducibility $\le_{JT}$ see \cite[8.4.13]{Nies:book}. We say that $A$ is JT-hard if $\Halt \le_{JT} A$. 

\begin{lemma}\label{lem:kolmo-jthard}
The following are equivalent for any set $A$.\\
(i) $A$ is JT-hard\\
(ii) There exists a computable order $h$ such that $K^A \leq^+ h(K^{\emptyset'})$
\end{lemma}

\begin{proof}
First suppose that $A$ is JT-hard. Consider the universal oracle machine $\UM$ for prefix complexity and consider the $\emptyset'$-partial computable function $S:\sigma \mapsto \UM^{\emptyset'}(\sigma)$. By definition of JT-hardness, there exists a computable order~$g$ and a family $(T_\sigma)$ of uniformly $A$-c.e.\ finite sets such that $\UM^{\emptyset'}(\sigma) \in T_\sigma$ and $|T_\sigma| \leq g(|\sigma|)$. Let now $x$ be any string and set $K^{\emptyset'}(x)=n$. This by definition means that $x=\UM^{\emptyset'}(\sigma)$ for some $\sigma$ of length~$n$. Relative to $A$, $x$ can be described by $\sigma$, and its index in the $A$-enumeration of $T_\sigma$. Thus $K^A(x) \leq^+ 2|\sigma|+2|T_\sigma| \leq^+ 2n+2g(n) \leq^+ h(K^{\emptyset'}(x))$, where $h(n)=2n+2g(n)$, which is a computable order, as wanted.

Conversely, suppose $K^A \leq^+ h(K^{\emptyset'})$ for some computable order~$h$. Let $f$ be a given $\emptyset'$-partial computable function. By definition of $K$, we have $K^{\emptyset'}(f(n))\leq^+ K^{\emptyset'}(n) \leq^+ 2 \log n$ and thus by assumption $K^A(f(n)) \leq h(2 \log n+c)$ for some constant $c$. To get an $A$-trace for $f(n)$, it thus suffices to $A$-enumerate all $x$'s such that $K^A(x) \leq h(2 \log n+c)$, and we know that there are at most $g(n)=2^{h(2 \log n+c)}$. Thus all $\emptyset'$-partial computable functions have an $A$-traced with size bounded by $g+O(1)$, which precisely means that $A$ is JT-hard. 

\end{proof}

\begin{theorem}
If $A$ is a superhigh c.e.\ set, then for any $\Delta^0_2$ DNR set $X$, one has $A \oplus X \geq_T \emptyset'$.  
\end{theorem}

\begin{proof}
It is known \cite[Thm.\ 8.4.16 and Cor.\ 8.4.27]{Nies:book} that for c.e.\ sets, superhighness is equivalent to JT-hardness. Thus we can apply Lemma~\ref{lem:kolmo-jthard} to get a computable order $h$ such that $K^A \leq^+ h(K^{\emptyset'})$. Now, by  relativizing Theorem~\ref{thm:kolmo-criterion} to $A$, we immediately obtain that   $A \oplus X \geq_T \emptyset'$.
\end{proof}

In the following we analyze the hypothesis of Theorem~\ref{thm:kolmo-criterion} and show it is equivalent to a property we call fairly highness. 


\begin{prop} \label{prop:strange_tracing} The following are equivalent  for a $\DII$ set $A$. 

(i)  $A$ is high

(ii)  For each pc functional $\Gamma $,  we can trace $\Gamma^{\ES'}$  by an $A$-c.e.\ trace with a finite bound. 

\end{prop}

(i)$\to$(ii):  $\ES'$ is low relative to $A$. Hence it has an enumeration relative to $A$ such that $\Gamma ^{\ES'}(x)$ becomes undefined only finitely often.

(ii) $\to $(i): Let $\Gamma^Z(x)$ be the number of steps it takes $J^Z(x)$ to converge (may be undefined). There is an $A$-c.e.\ trace $T_x$ for $\Gamma^{\ES'}$. Let $M(x)= \max T_x$, then $M \leT A'$. Hence  $\ES'' \leT A'$.

The  hypothesis of Theorem~\ref{thm:kolmo-criterion} is equivalent to the   following  condition on a set~$A$.

\begin{df}  \label{df:strange_tr}  $A$ is \emph{medium high} if  there exists a $\emptyset'$-computable order $h$ such  that $A$ c.e.--traces with bound $h$ every $\emptyset'$-partial computable function. 

\end{df}

 For a c.e.\ set $S$, being medium low (see   Section~\ref{s:fairly low}) is equivalent to that $J^S$ can be traced with a  bound computable in $S$. 
Hence, by pseudo jump inversion (letting $S \equiv_T \Halt$), if 
 $A$ is $\DII$, then  by Theorem~\ref{thm:proper_low},  $A$ is medium high iff $\Halt$ is medium low relative to $A$.  In particular  there  is a c.e.\ set $A$   that is medium high but  not superhigh.
Thus,  for $\Delta^0_2$ sets,  this property  lies properly  between superhighness and highness.

\subsection{An alternative recursion-theoretic proof}

\label{ss:not_K}
\newenvironment{changemargin}[2]{%
\list{}{\rightmargin#2\leftmargin#1
\parsep=0pt\topsep=0pt\partopsep=0pt}
\item[]}
{\endlist}
\newenvironment{indentmore}{\begin{changemargin}{0.5cm}{0cm}}{\end{changemargin}}


\n {\it Notation}: 
For an expression $E$ which is approximable during stages $s$, we denote by $E[s]$ its value by the end 
of stage $s$. 
Let  $m_x(n)$ be the modulus function of $A = W_x$, computable from $A$
Further, let  $A'[s]$ be the set of those $i$ for which $\Phi_{i,s}(A \uhr{s})(i)$ is convergent. Thus, $A'(y)[s] = 1$  iff $y \in A'[s]$.  
Note that $A'[s]$ is an approximation to $A'$ relative to $A$. 
We analogously denote  by $(\sigma)'[s]$ (for $\sigma \in 2^{< \omega}$) the set of those $i$ for which $\Phi_{i,s}(\sigma)(i)$ is convergent, 
and, consequently, $(\sigma)'(y)[s] =1$ iff $y \in (\sigma)'[s]$.

\begin{theorem} If $A$ is a nonlow$_1$ c.e.\ set then $A$ joins to $\Halt$ all $\Delta^0_2$ DNC functions.
\end{theorem}

\begin{proof}  
The main idea is to use a permitting argument at $\Halt$-level. 
When $e \in \Halt$ we can eventually verify that at some step $s$
and simultaneously indicate that up to step  $s$ all relevant approximations to $f$ at some arguments (see later)  are still not stable and  equal to their final values, 
while when $e \notin \Halt$, 
since $A' >_T \Halt$,  $A'$ eventually has to permit a situation when approximations to $f$ at some argument are already stable and equal to a final value. 
We  substantially use DNC-ness of a given $\Delta^0_2$ function  $f$ to do that.

\n Let $f$ be a DNC $\Delta^0_2$ function,  $A = W_x$ a nonlow$_1$ c.e.\ set  
and $F$ a computable function such that $\lim F(y,s) = f(y)$. 

\n For $\sigma \preceq A[s]$  for some $s$, 
$t(\sigma)$ denotes the least $j$ for which $A[j] \uhr{|\sigma|} = \sigma $.
Note that if $\sigma \prec A$ then $t(\sigma) = m_x(|\sigma|)$.

\n We use Recursion Theorem to get for any $e$ (uniformly in $e$)
\begin{indentmore}
- indices of partial computable functions $a(e, i)$  such that  
$J(a(e, i)){\downarrow}$  if and only if $e \in \Halt$, and if $e \in \Halt$  properties described  below hold. 
\end{indentmore}  

\n {\it Case} $e \in \Halt$:  let $e \in \Halt[\texttt{at} \: s_0]$,  $\eta = {A \uhr{s_0}[s_0]}$.
For $0 \leq j \leq s_0$ let $\tau_j = {A \uhr{j}{[s_0]}}$.
Make first $J(a(e, i)){\downarrow}$ for all   $a(e, i)$ and, second, 
for those $i < s_0$ for which $(\tau_j)'(i)[t(\tau_j)] = 1$   for some $i < j \leq s_0$, take the least such $j$, say $j_0$, and
make $J(a(e, i)) = F(a(e, i), t(\tau_{j_0}))$  
(an output value of others $J(a(e, i))$ is not relevant).

\n {\it Claim.} $\Halt \leq_T A \oplus f$. 
Given $e$, using $A \oplus f$, search for the least $s$ such that

\n - either $e \in \Halt[s]$, so that $e \in \Halt$,    or\\
\n - for some $i < s$, $\sigma_0 \preceq A \uhr{s}$, $i < |\sigma_0|$ we have:
$\sigma_0$ is the shortest $\sigma \preceq A \uhr{s}$ for which $(\sigma)'(i)[t(\sigma)] = 1$,  $i < |\sigma|$ and  
$F(a(e, i), t(\sigma_0)) = f(a(e, i)) $, in which case $e \notin \Halt$.

\n {\it Verification}: 
\n Let $e \in \Halt[\texttt{at} \: s_0]$. 
We have to show that we cannot in steps $s < s_0$ mistakenly decide $e \notin \Halt$.  
Take all $i < s_0$ for which there is $\sigma_j = A \uhr{j}$ such that $i < j$ and $(\sigma_j)'(i)[t(\sigma_j)]  = 1$
and for those $i$'s let $k(i)$ denote the least such $j$.  We could mistakenly decide $e \notin \Halt$ only when 
$t(\sigma_{k(i)})  < s_0$. But for such $\sigma_{k(i)}$  necessarily $\sigma_{k(i)} \prec \eta = A \uhr{s_0}[s_0] $  
(i.e. $\sigma_{k(i)} = \tau_{k(i)}$)  
and we have prevented to do a mistake by making $J(a(e, i)) = F(a(e, i), t(\sigma_{k(i)}))  $  
thereby forcing $ f(a(e, i))  \neq J(a(e, i))  =  F(a(e, i), t(\sigma_{k(i)})) $.

\n If $e \notin \Halt$, then there is a step $s$ at which we can make a decision $e \notin \Halt$. 
To see it let $H(x)$ be the $\Halt$-computable modulus function of the $\lim F(x,s) = f(x)$.
Then there is an $i$ such that $i \in A'[\texttt{at} \: t]$ and  $t > H(a(e, i))$, since otherwise the function 
$H(a(e, . ))$ would dominate the $A$-modulus of $A'$, a contradiction with nonlowness of $A$.
This together with the fact that $F(a(e, i), k) = f(a(e, i)) $ for all $k \geq H(a(e, i))$ finishes the proof.

\end{proof}

\section{Medium lowness}

\label{s:fairly low}
 \Kuc\  and Nies  worked in Prague, May. Later their work was improved by Faizramonov.

Recall that  a set $A$ \emph{superlow} if there is a computable function $g$ such that $A'(x) = \lim_s g(x,s)$ with the number of changes for $x$ bounded by a computable function $h$.

We now partially relativize to $A$ the superlowness of $A$  itself.

\begin{definition} We call a set $A$ \emph{medium low} if there is a computable function $g$ such that $A'(x) = \lim_s g(x,s)$ with the number of changes for $x$ bounded by a function $h \leT A$. \end{definition}

Clearly, we have the implications
 
 \bc superlow $\RA$ medium low $\RA$ low. \ec
 
These are proper implications.
\begin{theorem}  \label{thm:proper_low} 

\
\n (i) There is a c.e.\ medium low set  $A$ that is not superlow. (Faizramonov)

\n (ii)  There is a c.e.\ low set that is not medium low.\end{theorem}

The following was introduced  in \cite[Def.\ 27]{Figueira.Hirschfeldt.ea:nd}
\begin{definition} \label{def:omjpdom} {\rm A set $S$ is \emph{$\omega$-c.e.-jump dominated} if there is an $\omega$-c.e.\ function $g(x)$ such that $J^S(x) \le g(x)$ for every~$x$ such that $J^S(x)$ is defined.
 }
 \end{definition}
 \cite{Figueira.Hirschfeldt.ea:nd} showed that superlowness implies $\omega$-c.e.-jump dominated. The converse implication holds for r.e.\ sets but not in general, for instance because each Demuth random set is $\omega$-c.e.-jump dominated by the proof   of \cite[Thm.\ 3.6.26]{Nies:book}.
 We can also define a partial relativization of being $\omega$-c.e.-jump dominated, where $g$ is only $\omega$-c.e.\ by $S$. The same relationships hold   for superlow vs. $\omega$-c.e.-jump dominated.

\newpage

\section{Randomness Zoo (Antoine Taveneaux)}

 \begin{center}

\begin{tikzpicture}[node distance=9mm, xscale=.05 , yscale=.05]
  \tikzstyle{place}=[circle,thick,draw=black!89,fill=blue!00,minimum size=17mm ]

\node (pi11r) [place] {$\Pi^1_1 R$};
\node (pi11mlr) [place, right = of pi11r , yshift=-0mm, xshift=-4mm] {$\Pi_1^1 MLR$}; 
\node (delta11) [place, right = of pi11mlr , yshift=-0mm, xshift=-4mm] {$\Delta_1^1 R$}; 
\node (dots1) [ right = of delta11 , yshift=-0mm, xshift=-5mm] {$ \dots$};
\node (3mlr) [place, right = of dots1, yshift=-0mm, xshift=-5mm ] {$3MLR$}; 
\node (2mlr) [ place,  right = of 3mlr, yshift=-0mm, xshift=-4mm ] {$2MLR$}; 
\node (limr) [text width=12mm, place,  below left  = of 2mlr, yshift=1mm, xshift=-16mm ] {$~~SR'$ \tiny{$=LimitR$}}; 
\node (demr) [place,  right= of limr, yshift=-9mm, xshift=-3mm] {$DemR$};
\node (wdemr) [place, below left  = of demr, yshift=8mm, xshift=-5mm] {$WDemR$};
\node (w2r) [place,  left= of limr, yshift=-0mm ,  yshift=-9mm, xshift=2mm] {\large{$W2R$}};
\node (balancedr) [place, below left= of wdemr, yshift=10mm , xshift=-4.4mm] {\scriptsize{$BalancedR$}};
\node (diffr) [place, below right= of balancedr, yshift=9mm , xshift=5.5mm] {$DiffR$};
\node (mlr) [place, below  = of diffr,  yshift=15mm , xshift=30mm  ] {\huge{$MLR$}};
\node (klr) [place, below left= of mlr, yshift=8mm, xshift=0mm ] {$KLR$};
\node (pinjr) [place, left= of klr, yshift=-0mm , xshift=5mm ] {$PInjR$};
\node (injr) [place, left= of pinjr, yshift=-13mm , xshift=-1mm ] {$InjR$};
\node (permr) [place, below left= of pinjr, yshift=0.5mm , xshift=9mm ] {$PermR$};
\node (pcr) [place, below left = of permr , yshift=4mm , xshift=2mm ] {$PCR$};
\node (fboundr) [place, left= of mlr, yshift=-3mm, xshift=-33mm ] {\scriptsize{$\text{FBound}R$}};
\node (cboundr) [place, below left= of fboundr, yshift=-15mm, xshift=-15mm ] {\scriptsize{$\text{CBound}R$}};
\node (cr) [place, below left = of pcr, yshift=9mm, xshift=0mm  ] {\large{$CR$}};
\node (klstoch) [place, below right  = of klr, yshift=-0mm, xshift=5mm ] {$KLStoch$};
\node (pchstoch) [place, below  left = of klstoch, xshift=-6mm, yshift=3mm ] {\tiny{$\text{MWC}Stoch$}};
\node (chstoch) [place, below left = of pchstoch, xshift=-2mm, yshift=8mm ] {\scriptsize{$\text{Ch}Stoch$}};
\node (sr) [place, below  = of cr,xshift=-0mm,yshift=-0mm] {\large{$SR$}};
\node (wr) [place, below  = of sr, xshift=-10mm, yshift=5mm] {$WR$};
\node (polyr) [place, below right   = of cr , xshift=2mm, yshift=-7.5mm] {$PolyR$};
\node (dimcomp1) [place, below   = of chstoch,  yshift=6.5mm] {\footnotesize{$dim_{comp}^{1} R $}};
\node (cdim1) [place, below   = of pchstoch, xshift=14mm,  yshift=6.5mm] {\footnotesize{$Cdim ^{1}R $}};
\node (dimcomps) [place, below   = of dimcomp1,  yshift=6.5mm] {\footnotesize{$dim_{comp}^{s} R $}};
\node (cdims) [place, below   = of cdim1,  yshift=6.5mm] {\footnotesize{$Cdim ^{s}R $}};
\node (dimcomps') [place, below   = of dimcomps,  yshift=6.5mm] {\footnotesize{$dim_{comp}^{s'} R $}};
\node (cdims') [place, below   = of cdims,  yshift=6.5mm] {\footnotesize{$Cdim^{s'} R $}};


\path (pi11r) edge[->] (pi11mlr);
\path (pi11mlr) edge[->] (delta11);
\path (delta11) edge[->] (dots1);
\path (dots1) edge[->] (3mlr);
\path (3mlr) edge[->] (2mlr);
\path (2mlr) edge[->] (limr);
\path (limr) edge[->] (demr);
\path (limr) edge[->] (w2r);
\path (w2r) edge[->] (wdemr);
\path (wdemr) edge[->] (balancedr);
\path (balancedr) edge[->] (diffr);
\path (demr) edge[->] (wdemr);
\path (diffr) edge[->] (mlr);
\path (mlr) edge[->] (klr);
\path (mlr) edge[->] (fboundr);
\path (fboundr) edge[->] (cboundr);
\path (cboundr) edge[->] (wr);
\path (klr) edge[->] (pinjr);
\path (pinjr) edge[->] (injr);
\path (injr) edge[->] (cr);
\path (pinjr) edge[->] (permr);
\path (permr) edge[->] (pcr);
\path (pcr) edge[->] (cr);
\path (pcr) edge[->] (pchstoch);
\path (cr) edge[->] (sr);
\path (cr) edge[->] (polyr);
\path (sr) edge[->] (wr);
\path (cr) edge[->] (chstoch);
\path (klr) edge[->] (klstoch);
\path (klstoch) edge[->] (pchstoch);
\path (klstoch) edge[->] (cdim1);
\path (pchstoch) edge[->] (chstoch);
\path (chstoch) edge[->] (dimcomp1);
\path (cdim1) edge[->] (dimcomp1);
\path (cdim1) edge[->] node[anchor=east] {\footnotesize{$s<1$}}  (cdims);
\path (cdims) edge[->]node[anchor=east] {\footnotesize{$s'<s$}} (cdims');
\path (dimcomp1) edge[->] node[anchor=east] {\footnotesize{$s<1$}} (dimcomps) ;
\path (dimcomps) edge[->]node[anchor=east] {\footnotesize{$s'<s$}}   (dimcomps');
\path (cdims') edge[->] (dimcomps');
\path (cdims) edge[->] (dimcomps);


\end{tikzpicture}

\end{center}

A full version is available at  \href{http://calculabilite.fr/randomnesszoo.pdf}{\texttt{http://calculabilite.fr/randomnesszoo.pdf}}. 

\vsp

\scalebox{.12}{\includegraphics{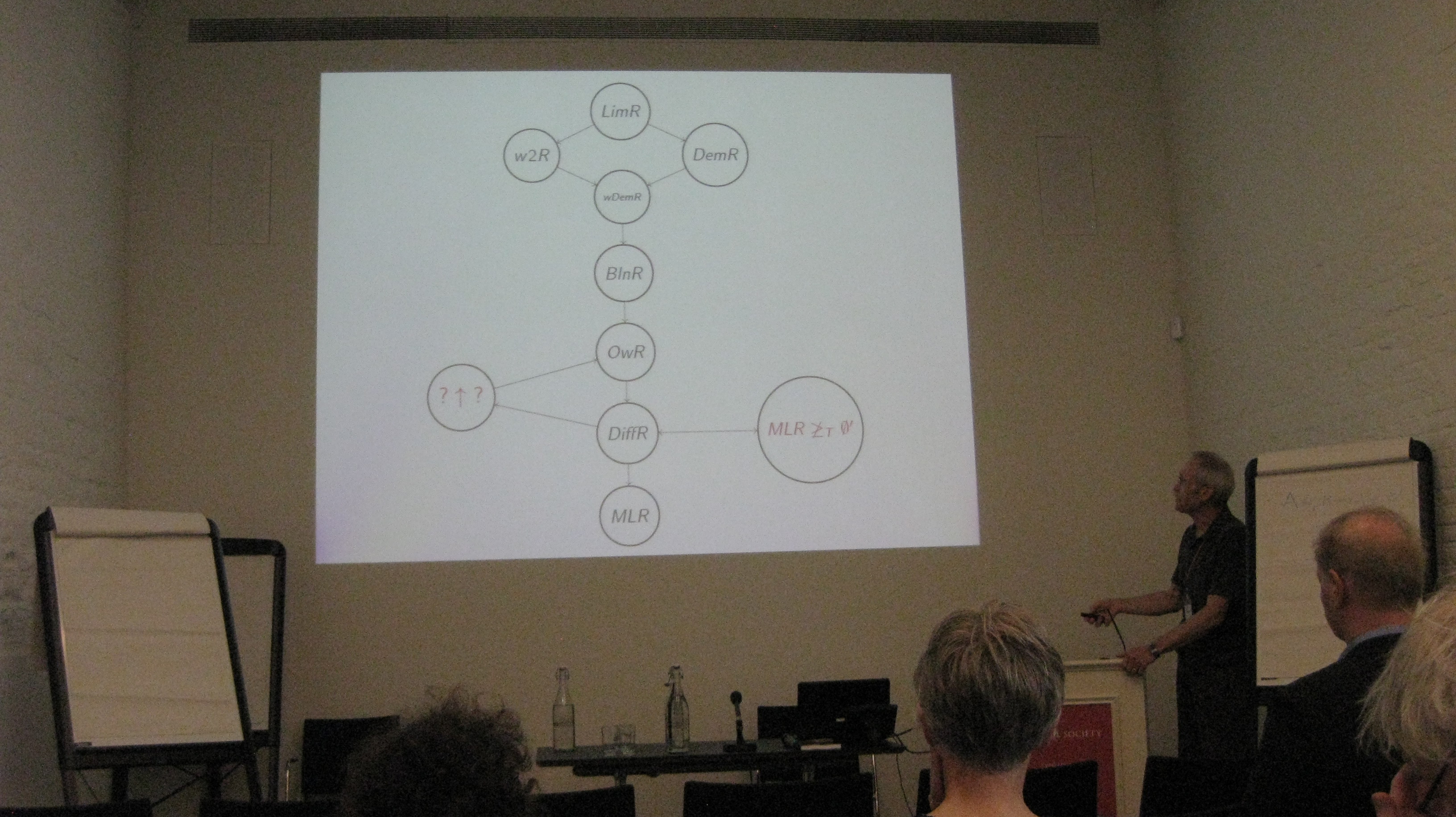}}

\bc Tonda \Kuc\ at the Incomputable conference in \\ Chicheley Hall, June 2012. \ec

\begin{df}[\textbf{MLR}: Martin-L\"of~Randomness (Definition 3.2.1 in \cite{Martin-Lof:66})]

\begin{enumerate}

\item
A  Martin-L\"of~test, a ML-test for short,  is a uniformly c.e. sequence $(G_m)_{m \in \NN }  $ of open sets such that $\forall n ~ \mu \left(  G_m \right)
 \leq 2^{-m}$. 

\item A set $R$ fails the test if $R \in \cap_{m \in \NN } G_m $ otherwise $R$ passes the test.

\item $R$ is Martin-L\"of~Random if $R$ passes each ML-test.

\end{enumerate}

\end{df}

\subsection{Weaker than MLR}

\begin{df}[Scan rule function]

For a  partial  function $ f : \fs \rightarrow \{\text{scan} , \text{select} \} \times \NN $  we denote $n: \fs \rightarrow  \NN$ and $\delta : \fs \rightarrow \{\text{scan} , \text{select} \}$. 

And we say that $f $ is a  scan rule  if for all $\sigma, \rho \in \fs $ such that $\sigma \prec \rho $ we have $n( \sigma) \not= n (\rho) $. 

The sequence of string observed by $f$ on $Z\in cs $ is defined by(if $f$ is well defined on this points): 
\[ V_f^A (0) = A (n(\varepsilon) )
\]
\[ V_f^A (k+1) = V_f^A (k). A (n(V_f^A (k)) )
\]
and  bits selected by $f$ are: 
\[ T_f^A (0) = \varepsilon \] 
\[
T_f^A (n+1) = \left\{
 \begin{array}{ll}
 T_f^A (n) & \mbox{if } \delta( V_f^A (n) )=\text{scan} \\
 T_f^A (n) . V_f^A (n+1) & \mbox{if } \delta( V_f^A (n) )=\text{select}
 \end{array}
\right.
\]
and we say that $f$ is well defined on $Z$ if $V_f^A (k)$ is well define for all $k$ and $(T_f^A (n))$ converge to an infinite string and we denote this infinite string by $T_f^A$.

\end{df}

\begin{df}[Martingale]

A martingale is a function $\mathcal{M } : \fs \rightarrow \RR^+ \cup {0}$ such that for all $\sigma \in \fs $
\[
\mathcal{M } ( \sigma ) = \frac{\mathcal{M }(\sigma 0 ) + \mathcal{M }(\sigma 1 ) }{2}
\]

We say that $\mathcal {M} $ succeed on $Z \in \cs $ if 
\[
\limsup_{n \rightarrow \infty } \mathcal {M} (Z  \upharpoonright n) =\infty
\]

\end{df}

\begin{df}[\textbf{KLR}: Kolmogorov–Loveland randomness (Definition  in \cite{Kolmogorov:63} and \cite{Loveland:66})]

$R\in \cs $ is KL-random if for any partial computable scan rule function  $f$ and any partial computable martingale $\mathcal{M}$ such that $f$ is well defined on $R$ the martingale martingale $\mathcal {M}$ does not succeed on $T_f^A$. 

\end{df}

\begin{df}[\textbf{KLStoch}: Kolmogorov–Loveland stochasticity (Definition  in \cite{Loveland:66})]
Let $\#0 : \fs \rightarrow \NN $  the function giving the number of ``0" in a string. 

A sequence $R$ is Kolmogorov–Loveland stochastic if for all partial computable scan rule $f $ such that $f$ is well defined on $Z$ we have:
\[
\lim_{n\rightarrow \infty } \frac{\#0  \left( T_f^A \upharpoonright n \right) }{n} = \frac{1}{2}
\]

\end{df}

\begin{df}[\textbf{MWCStoch}: Mises-Wald-Church stochasticity (Definition  in \cite{Mises:19} and \cite{Church:40})]

A sequence $R$ is   Mises-Wald-Church stochastic if for all  partial computable monotonic scan rule function $f$ is well defined on $Z$ we have:
\[
\lim_{n\rightarrow \infty } \frac{\#0  \left( T_f^A \upharpoonright n \right) }{n} = \frac{1}{2}
\]
\end{df}

\begin{df}[\textbf{ChStoch}: Church stochasticity (Definition  in \cite{Mises:19} and \cite{Church:40})]
A sequence $R$ is   Church stochastic if for all  total computable monotonic scan rule function $f$ we have:
\[
\lim_{n\rightarrow \infty } \frac{\#0  \left( R_{f(0)}R_{f(1)} \dots R_{f(n)}  \right) }{n} = \frac{1}{2}
\]

\end{df}

\begin{df}[\textbf{PInjR}: partial injective randomness (Definition  in \cite{Miller.Nies:06})]
A sequence $R$ is partial injective random if for any total computable injective function $g: \NN \rightarrow \NN$ and any partial computable martingale $\mathcal{M}$ this martingale is defined and does not succeed on the sequence $R_{f(1)} R_{f(2)} \dots R_{f(n)} R_{f(n+1)} \dots $. 

\end{df}

\begin{df}[\textbf{InjR}: injective randomness (Definition  in \cite{Miller.Nies:06,Bienvenu.Hoelzl.ea:09})]
A sequence $R$ is  injective random if for any total computable injective function $g: \NN \rightarrow \NN$ and any total computable martingale $\mathcal{M}$ this martingale does not succeed on the sequence $R_{f(1)} R_{f(2)} \dots R_{f(n)} R_{f(n+1)} \dots $. 
\end{df}

\begin{df}[\textbf{PermR}: partial permutation randomness (Definition  in \cite{Bienvenu.Hoelzl.ea:09})]
A sequence $R$ is partial permutation random  if for any total computable bijective function $g: \NN \rightarrow \NN$ and any partial computable martingale $\mathcal{M}$ this martingale is defined and does not succeed on the sequence $R_{f(1)} R_{f(2)} \dots R_{f(n)} R_{f(n+1)} \dots $. 
\end{df}

\begin{df}[\textbf{PCR}: partial computable randomness (Definition  in \cite{Ambos:97})]
A sequence $R$ is partial computable random  if for all  partial computable martingale $\mathcal{M} $ if $\mathcal {M } (R\upharpoonright n ) $  is define for all $n$ and $\mathcal {M }$ does not succeed on $R$. 

\end{df}

\begin{df}[\textbf{CR}: computable randomness (Definition  in \cite{Schnorr:75})]
A sequence $R$ is  computable random  if for all  total computable martingale $\mathcal{M} $ this martingale succeed on $R$. 
\end{df}

\begin{df}[\textbf{SR}: Schnorr randomness(Definition  in \cite{Schnorr:75})]

 A  Schnorr~test  is a uniformly c.e. sequence $(G_m)_{m \in \NN }  $ of open sets such that $\forall n ~ \mu (G_m) = 2^{-m}$.

 $R$  is Schnorr random if for any Schnorr~test $(G_m)_{m \in \NN }  $ $R \not\in \cap_{m \in \NN } G_m $
\end{df}

\begin{df}[\textbf{FBoundR}: finitely bounded randomness (Definition  in \cite{Brodhead.Downey.ea:12})]
$R $ is finitely bounded random if $R$ passes any Martin-L\"of~test $(U_n)$ such that for every $n$, $\# U_n < \infty $ (with $\# U_n$ the number of sting enumerated 
in $U_n$).  
\end{df}

\begin{df}[\textbf{CBoundR}: computably bounded randomness (Definition  in \cite{Brodhead.Downey.ea:12})]
A Martin-L\"of~test $(U_n) $ is computably bounded if there is some
total computable function $f$ such that $\# U_n \leq f(n) $ for every $n$.

$R$ is computably bounded random if $R$ passes every computably bounded Martin-L\"of~test. 
\end{df}

\begin{df}[\textbf{WR}: weakly randomness (Definition  in \cite{Kurtz:81})]
$R$ is weakly random if $R\in U$ for every   $\Sigma_1^0 $ set $U \subseteq \cs $ of measure 1. 
\end{df}

\begin{df}[\textbf{PolyR}: polynomial randomness (Definition  in \cite{Wang:96})]
A sequence $R$ is polynomially random if any martingale computable in polynomial time $\mathcal{M}$ do not succeed on $R$.  
\end{df}

\begin{df}[\textbf{$\text{dim}^s_{\text{comp}}   $R}: computable $s$-randomness (Definition  in \cite{Lutz:03} and \cite{Mayordomo:02})]
A  computable~$s$-test  is a uniformly computable sequence $(G_m)_{m \in \NN }  $ of computable open sets such that for all $n$
\[ 
\sum_{ x \in G_m }2^{-s| x|} \leq  2^{-n}.
\]

 $R$  is computably $s$-random if for all  $s' < s$   and computable~$s$-tests $(G_m)$ we have: 
 \[
 R \not\in \bigcap_{m \in \NN} G_m
 \]

\end{df}

\begin{df}[\textbf{$\text{Cdim}^s$R}: constructive $s$-randomness (Definition  in \cite{Lutz:03} and \cite{Mayordomo:02})]
A  constructive~$s$-test  is a uniformly computable sequence $(G_m)_{m \in \NN }  $ of computable enumerable open sets such that for all $n$
\[ 
\sum_{ x \in G_m }2^{-s| x|} \leq  2^{-n}.
\]

 $R$  is computably $s$-random if for all  $s' < s$   and computable~$s$-tests $(G_m)$ we have: 
 \[
 R \not\in \bigcap_{m \in \NN} G_m
 \]

\end{df}

\subsection{Stronger than MLR}

\begin{df}[\textbf{DiffR}: difference randomness (Definition  in \cite{Franklin.Ng:10})]

A difference test is given by a sequence $(V_m)_{m \in \NN } $ of uniformly c.e. sets and a $\Pi_1^0 $ set $P$ 
such that $\mu (P \bigcap V_m ) \leq 2^{-m}$ for every $m$.

A sequence $R$ is difference random if  for any difference test $(V_m)_{m \in \NN } , P$  we have  
\[
R \not\in P\bigcap \left( \cap_{m \in \NN } V_m  \right). 
\] 



\end{df}

\begin{df}[\textbf{BalancedR}: balanced randomness (Definition  in \cite{Figueira.Hirschfeldt.ea:nd})]
A balanced test is a sequence $(V_m)_{m \in \NN } $ of c.e.  sets such that  $V_i = W_{f(i)}$ for some $2^n$-c.e. function $f$ and $\mu (V_m ) \leq 2^{-m} $for every $m$. 

 A sequence $R$ is  balanced random if $R$ passes any balanced test.
\end{df}

\begin{df}[\textbf{WDemR}: weak Demuth randomness  (Definition  in \cite{Demuth:82a})]
 Demuth test is a sequence $(V_m)_{m \in \NN } $ of c.e.  sets such that  $V_i = W_{f(i)}$ for some $\omega$-c.e. function $f$ and $\mu (V_m ) \leq 2^{-m} $for every $m$. 
 
 A sequence $R$ is weak Demuth random if for any Demuth test $(V_m)$ we have $R \not\in \cap_{m \in \NN} V_m$. 
\end{df}

\begin{df}[\textbf{DemR}: Demuth randomness (Definition  in \cite{Demuth:82a})]
A Demuth test is a sequence $(V_m)_{m \in \NN } $ of c.e.  sets such that  $V_i = W_{f(i)}$ for some $\omega$-c.e. function $f$ and $\mu (V_m ) \leq 2^{-m} $for every $i$. 

A sequence $R$ is Demuth random if for any Demuth test $(V_m)$ we have $R \not\in V_m $ for almost all $m$. 

\end{df}

\begin{df}[\textbf{W2R}: weak 2-randomness (Definition  in \cite{Kautz:91})]
A sequence $R$ is weak 2-random if $R \not\in U $ for every $\Pi_2^0 $ set $U \subset \cs $ of measure 0.

\end{df}

\begin{df}[\textbf{LimitR}: limit randomness (Definition  in \cite{Kucera.Nies:11})]
A limit test is a sequence $(V_m)_{m \in \NN } $ of c.e.  sets such that  $V_i = W_{f(i)}$ for some $\Delta_2^0 $-computable function $f$ and $\mu (V_m ) \leq 2^{-m} $for every $m$. 

  A sequence $R$ is  limit random if for any limit test $(V_m)$ we have $R \not\in V_m $ for almost all $m$. 

\end{df}

\begin{df}[\textbf{$\Delta_1^1$R}: $\Delta_1^1$ randomness (Definition  in \cite{MartinLof:68})]
$R$ is $\Delta_1^1$-random if $R$ avoids each null $\Delta_1^1$-class. 
\end{df}

\begin{df}[\textbf{$\Pi_1^1$MLR}: $\Pi_1^1$-Martin-L\"of~Randomness (Definition  in \cite{Hjorth.Nies:07})]
A $\Pi_1^1 $-Martin-L\"of~test is a sequence $(G_m)_{m\in \NN }$ of open sets such that $\forall n~~ \mu (G_m ) \leq 2^{-m} $ and the relation $\{ \langle  m , \sigma \rangle |[\sigma ] \subseteq G_m  \}$ is $\Pi_1^1$

$R$ is $\Pi_1^1$-Martin-L\"of~Random if $R$ passes each $\Pi_1^1$-ML-test.
\end{df}

\begin{df}[\textbf{$\Pi_1^1$R}: $\Pi_1^1$-Randomness (Definition  in \cite{Hjorth.Nies:07})]
$R$ is $\Pi_1^1$-random if $R$ avoids each null $\Pi_1^1$-class. 
\end{df}

\section{Complexity of recursive splittings of random sets}

\noindent
Ng, Nies and Stephan investigated random sets which preserve as much
information as possible
when subjected to a recursive splitting. In each high Turing degree we
build a Schnorr random set such that one can preserve the Turing degree
in both halves of the splitting by any infinite and co-infinite 
recursive set. We build a Martin-L\"of random set so that one
half can still preserve most of the information because it is Turing 
above $\Omega$. Due to the Theorem of van Lambalgen, some information is
necessarily lost.

\subsection{Schnorr random sets}

\begin{thm} 
Every high Turing degree contains a Schnorr random set $Z$ such that
$Z \leq_T R \cap Z$ for every infinite
r.e.\ set $R$. Thus, if $R$ is recursive then $Z \equiv_T R \cap Z$.
\end{thm}

\begin{proof}
Let $g$ be a function in the given Turing degree which dominates all
recursive functions. Now one can define a
further dominating function $f$ such that $f(n)$ is the sum of all
$\varphi_k(n)$ where $k \leq n$ and
$\varphi_k(n)$ is computed within $g(n)$ computation steps.
Furthermore, one can inductively define for each $n$ and
each $k=0,1,\ldots,n$ that $E_{n,k}$ is the set
of the first $2n+5$ elements of
$$
   A = W_{k,g(f(n))} - \{0,1,\ldots,f(n)\} -
   \bigcup_{(n',k'): (n',k') <_{lex} (n,k) \wedge k' \leq n'} E_{n',k'};
$$
whenever they exist; if they do not exist then $E_{n,k} = \emptyset$.
In the case that $E_{n,k}$ has $2n+5$ elements, the corresponding
entries are used as follows: Bit $0$ is used to code one
bit of $\{(x,y): g(x) = y\}$ in order to encode the graph of $g$;
bit $k+1$ is used to code whether $\varphi_k(n+1)$ contributes
to $f(n+1)$ or not (where $k=0,1,\ldots,n,n+1$),
bit $k+n+3$ is used to code whether $E_{n+1,k}$ is
empty (where $k=0,1,\ldots,n,n+1$). So in total
$2n+5$ coding bits are used. This determines how $Z$ is coded on these
entries. By \cite[Lemma 7.5.1]{Nies:book} there is a fixed 
$g$-recursive martingale $L$ with rational values which dominates
all recursive martingales up to a multiplicative constant.
The other entries of $Z$ are chosen so that $L$ does not increase.

The resulting set is Schnorr random, as $L$ has after up to $g(n)$ bets at
most the value $2^{(3n+2)^2}$; for being covered by this martingale in the
Schnorr sense, one could
require that some recursive martingale obtains this value infinitely
often already after $h(n)$ many
steps for some recursive function $h$. See Franklin and Stephan
\cite[Proposition 2.2]{Franklin.Stephan:10} for additional information.

Let $W_e$ be an infnite set. As $g$ dominates all recursive functions,
one can find for almost all $n$
more than $(2n+5)^3$ elements in $W_e$ above $g(n)$ in time $g(f(n))$.
Therefore, $E_{n,e}$ is non-empty for almost all $n$. So when starting
with a sufficiently large $n$, using $Z$ and the enumeration of $W_e$
one can find all the entries for $E_{m,e}$ with $m \geq n$, 
and therefore compute the function $g$ relative to $Z \cap W_e$. This
permits to compute $Z$. Thus $Z \leq_T Z \cap W_e$.
\end{proof}

\subsection{Martin-L\"of random sets}

\begin{thm} \label{th:ml}
There is a Martin-L\"of random set $Z$ such that $\Omega \leq_T Z \cap
R$ or $\Omega \leq_T Z - R$
for every recursive set $R$.
\end{thm}

\begin{proof}
The set $Z$ is constructed in several steps:
\begin{itemize}
\item A construction of an r-maximal set $S$ with complement $E_0 \cup
E_1 \cup E_2 \cup \ldots$
         where the parts $E_0,E_1,E_2,\ldots$ are finite sets with
$\max E_n < \min E_{n+1}$ which each
         maximise their e-state; this ensures that $S$ is $r$-maximal;
\item Letting $Z(a_k) = \Omega(k)$ for the $k$-th bit $a_k$ in the
complement of $S$;
\item Taking a set $P$ which is PA-complete and low for $\Omega$;
\item Defining the bits in $S$ according to a set which is random
relative to $P \oplus \Omega$.
\end{itemize}
Once $Z$ is constructed, it is shown that the resulting set $Z$ is
Martin-L\"of random and that for every
recursive splitting, one half of it is Turing above $\Omega$.

First the construction of the r-maximal set $S$ is given.
One defines the $n$-th e-state of a set $D$ as the sum
$3^n a_0+3^{n-1}a_1+\ldots 3a_{n-1}+a_n$ where each
$a_k$ is $2$ if $\varphi_k$ is defined on $D$ and takes always values
from $\{1,2,\ldots\}$ and $a_k$ is $1$ if $\varphi_k$ is defined on
$D$ and takes each time the value $0$ and $a_k = 0$ otherwise.
Furthermore, one can also consider an approximation to the e-state
at some given time.
Now one will make sets $E_n$, approximated as $E_{n,t}$ such that
$E_{n,t+1} \neq E_{n,t}$ such that their e-state has the parameter
$n$ as well and that $E_{n,t+1} \neq E_{n,t}$ only when some
$E_m$ with $m \leq n$ increases its e-state. Furthermore, when
the e-state of $E_n$ is $p$ then $E_n$ has $n^2 \cdot 2^{3^{n+1}-p}$
elements and $\max(E_{n,t}) < \min(E_{n+1,t})$ for all $n,t$.
Initially each $E_{n,0}$ is an interval of length
$2^{3^n} \cdot n^2$ and its e-state is $0$. All elements
below $t$ which are not on an interval $E_{n,t}$ at stage $t$ are
enumerated into $S$.

Second, let $E = E_0 \cup E_1 \cup E_2 \cup \ldots$ and let $a_k$ be
the $k$-th element of $E$ (in ascending order).
Now let $Z(a_k) = \Omega(k)$. Note that whenever $a_k \in E_n$ then
$a_{k+1} \in E_n \cup E_{n+1}$.

Third, Miller \cite{Downey.Hirschfeldt.ea:05,Miller:10}
showed that there is a set $P$ which is low for $\Omega$ and PA-complete. 

Fourth, let $V$ be a set which is Martin-L\"of random relative to $P
\oplus \Omega$
and for $x \in S$, let $Z(x) = V(x)$. This completes the construction of $Z$.

For the further proof, note that one can compute the list of members
of $E_{n+1}$ if one knows the list
of members of $E_0 \cup E_1 \cup \ldots \cup E_n$ as well as the final
e-state of
$E_{n+1}$: this is just done by simulating the construction until the
values of $E_{m,t}$
have stabilised as $E_m$ for $m \leq n$ and the e-state of $E_{n+1,t}$
has stabilised at
the corresponding value. More generally, the e-states of
$E_0,E_1,\ldots,E_{n+1}$
permit to compute the sets $E_0,E_1,\ldots,E_{n+1}$ explicitly. Now,
the number of bits
contained in $E_0 \cup E_1 \cup \ldots \cup E_n$ is at least $n^3/8$
while the $e$-states
of $E_0,E_1,\ldots,E_n,E_{n+1}$ together can be described with
$\log(3) \cdot \frac{n(n+1)}{2}$
bits. Therefore, the positions of $E_0,E_1,\ldots,E_n,E_{n+1}$ are
reached at some stage $t$
before the left-r.e.\ approximation of $\Omega$ reaches on $E_0 \cup
E_1 \cup \ldots \cup E_n$
its final values. This fact an used to compute $\Omega$ from $Z$ in an
iterative manner:
Knowing the bits of $\Omega$ coded on $E_0 \cup E_1 \cup \ldots \cup
E_n$ permits to
compute the position of $E_{n+1}$ which then again permits to look up
the bits of $\Omega$
coded on $E_{n+1}$ from $Z$. So $\Omega \leq_T Z$ by only taking into
consideration
the positions in $E$. As $S$ is r-maximal, it holds for any recursive
set $R$ that almost all
elements of $E$ are either inside $R$ or outside $R$; depending on
which case holds,
one can compute $\Omega$ from either $Z \cap R$ or $Z - R$.

The last part of the proof is to show that $Z$ is Martin-L\"of random.
Assume that this is not the
case. Then there is a $P$-recursive martingale $M$ succeeding on $Z$,
as $P$ is PA-complete.

If one relaxes $M$ to being partial-recursive, then one can in
addition permit that $M$ - while processing
the input from $Z$, computes the set $E_{n+1}$ whenever it has
processed all members of $E_0 \cup 
E_1 \cup \ldots \cup E_n$ and found the corresponding values of
$\Omega$ and used the time the
left-approximation takes on them to get the position of $E_{n+1}$.
Therefore the martingale knows
whenever it is betting on an entry of $Z$ belonging to $S$ or
belonging to $E$; that is, there is a partial-recursive function
$\gamma$ such that $\gamma(Z(0)Z(1)\ldots Z(n)) = S(n+1)$ and
$\gamma$ might be wrong or undefined if the input is not a prefix of $Z$.
The martingale evaluates $\gamma$ prior to betting. Now let
\begin{eqnarray*}
  q_n & = & \frac{M(Z(0)Z(1)\ldots Z(n)Z(n+1))}{M(Z(0)Z(1)\ldots Z(n))}; \\
  s_n & = & \prod_{m \in \{0,1,\ldots,n\} \cap S} q_m; \\
  e_n & = & \prod_{m \in \{0,1,\ldots,n\} \cap E} q_m.
\end{eqnarray*}
As the limit superior of $s_n \cdot e_n$ is $\infty$, it follows
that (a) the limit superior of the $s_n$ is $\infty$ or
(b) the limit superior of the $e_n$ is $\infty$.
Now it is shown that in both cases (a) and (b) a contradiction
can be derived; for this contradiction, note that by construction $\Omega$ is
Martin-L\"of random relative to $P$, $V$ is Martin-L\"of random
relative to $\Omega \oplus P$ and then by the Theorem of van Lambalgen
$\Omega$ is Martin-L\"of random relative to $V \oplus P$.

Case (a): the limit superior of the $s_n$ is $\infty$.
In this case,
one could make another modification of $M$ such that the resulting
martingale $\tilde M^{P \oplus \Omega}$ abstains from betting at $n$ when
$n \in E$ (but uses $\Omega$ to retrieve the value) and bets on
$V(n)$ when $n \in S$ (where $V(n) = Z(n)$). This permits to show
that $V$ is not Martin-L\"of random relative to $P \oplus \Omega$,
a contradiction.

Case (b): the limit superior of the $e_n$ is $\infty$.
In this case,
one could make another modification of $M$ such that the resulting
partial $P \oplus V$-recursive martingale $\tilde M^{P \oplus V}$
bets on $\Omega(n+1)$ the value $q_{a_{n+1}} = e_{a_{n+1}}/e_{a_n}$
where this share is computed from $Z(0)Z(1)\ldots Z(a_{n+1}-1)$
where in turn this string can be retrieved using $\gamma$ and
the oracle $P \oplus V$ and $\Omega(0)\Omega(1)\ldots \Omega(n)$.
Hence $\tilde M^{P \oplus V}$ takes on $\Omega(0)\Omega(1)\ldots \Omega(n)$
the value $e_n = q_{a_0} q_{a_1} \ldots q_{a_n}$ and therefore $\Omega$
is not Martin-L\"of random relative to $P \oplus V$, a contradiction.

This case-distinction then shows that $Z$, other than assumed, is
indeed Martin-L\"of random.
\end{proof}

\begin{remark}
This proof can be generalised such that for given set $Y$ there is a
Martin-L\"of random $Z$ such
that either $Y \leq_T Z \cap R$ or $Y \leq_T Z - R$ for every recursive set $R$.
\end{remark}

\noindent
The next small result shows that one can split every complete Turing
degree into Turing incomplete ML-random degrees. Ng, Nies and Stephan would like to thank
Yu Liang for a simplification of the proof of Theorem~\ref{th:yuliang}.

\begin{thm} \label{th:yuliang}
Every Turing degree above that of $\Omega$ contains a Martin-L\"of
random set of the form
$X \oplus Y$ such that $X$ is low and $Y$ is Turing incomplete.
\end{thm}

\begin{proof}
Let $Z \geq_T \Omega$ be in the given Turing degree and let $X$ be a low
Martin-L\"of random set, for example a half
of $\Omega$. Relativising the Theorem of Ku\v cera and G\'acs
to $X$ gives that there is a set $Y$ which is Martin-L\"of random relative
to $X$ and which satisfies $Z \equiv_T X \oplus Y$.
Now $X$ is Martin-L\"of random relative to $Y$
by the Theorem of van Lambalgen and therefore $X \not\leq_T Y$;
hence $Y$ is Turing incomplete.
\end{proof}

%
%

\section{Indifference for weak $1$-genericity}

 Frank Stephan and  Jason Teutsch answered   in the affirmative the open question 8.4 in  Adam Day's  \href{http://sigmaone.files.wordpress.com/2011/11/day_genericity.pdf}{paper} (the printed url  is

\noindent \verb|http://sigmaone.files.wordpress.com/2011/11/day_genericity.pdf| ).

\begin{theorem}
No 2-generic set $X$ can compute a set which is indifferent for $X$
with respect to weak $1$-genericity.
\end{theorem}

\begin{proof}
First, assume by way of contradiction that $X$ is $2$-generic and
$I^X$ is an algorithm which produces an infinite set for $X$ which
is indifferent for $X$ with respect to weak 1-genericity.

One can extend this definition to $I^\sigma$ being the elements which can
be enumerated into $I$ relative to $\sigma$ without querying the oracle
outside the domain of $\sigma$. Note that $I^X$ is the union over all
$I^\sigma$ with $\sigma \preceq I$.

Furthermore, one can now define an extension function $\psi$ such that
$\psi(\sigma)$ is the first extension $\tau$ of $\sigma$ found for which
there is an $x \in I^\tau$ with $|\sigma| < x < |\tau|$. This is a
partial-recursive extension function. Note that every prefix $\sigma$
of $X$ satisfies that $\psi(\sigma)$ is defined. Hence, by 2-genericity
there is a prefix $\eta \preceq X$ such that every extension $\sigma$
of $\eta$ is in the domain of $\psi$. Now one defines for any string
of the form $\eta \vartheta$ a function $f(\eta \vartheta)$ which is defined
inductively with $\vartheta_0,\vartheta_1,\ldots,\vartheta_n$ being all
strings of length $|\vartheta|$ and $\tau_0,\tau_1,\ldots,\tau_n$ being
chosen such that the following equations hold:
\begin{eqnarray*}
   \eta \vartheta_0 \tau_0 & = & \psi(\eta \vartheta_0); \\
   \eta \vartheta_{m+1} \tau_0 \tau_1 \ldots \tau_m \tau_{m+1} & = &
   \psi(\eta \vartheta_{m+1} \tau_0 \tau_1 \ldots \tau_m)
      \mbox{ for all $m<n$;} \\
   f(\eta \vartheta) & = & \eta \vartheta \tau_0 \tau_1 \ldots \tau_n.
\end{eqnarray*}
Note that in this definition, the extending part of $f(\eta \vartheta)$
only depends on the length of $\vartheta$ and not of the actual bits;
furthermore, $I^{f(\eta \vartheta)}$ contains an element $x$ such that
$|\eta \vartheta| < x < |f(\eta \vartheta)|$.

Now one partitions the natural numbers in intervals
$\{x: x<|\eta|\}$ and $J_0,J_1,\ldots$ and there is on each interval
$J_n$ a string $\kappa_n$ of length $|J_n|$
so that $f(\eta \vartheta) = \eta \vartheta \kappa_n$
whenever $\vartheta$ has the length $\min(J_n)-\min(J_0)$.

Now let for any $\sigma$ extending $\eta$ the
function $g(\sigma)$ be $f(\sigma 0^k)$ for the first $k$ such that
there is an $n$ with $|\sigma 0^k| = \min(J_n)$; if $\sigma$ is a prefix
of $\eta$ then let $g(\sigma) = f(\eta)$; if $\sigma$ is incomparable to $\eta$
then let $g(\sigma) = \sigma 0$. Note that $g$ enforces that every
weakly $1$-generic set with prefix $\eta$ equals to $\kappa_n$
on $J_n$ for some $n$.

Let $\sigma_n$ be the prefix of $X$ of length $\min(J_n)$.
Let $H$ be the set of all $n$ such that $X$ restricted to $J_n$
equals $\kappa_n$; note that for each $n \in H$ there is an
element $a_n \in I^X \cap J_n$. Now let $Y$ be the symmetric
difference of $X$ and $\{a_n: n \in H\}$. Note that the
set $Y$ does not coincide with $\kappa_n$ on $J_n$ for
any $n$. As $\eta$ is a prefix of $Y$,
$Y$ is not weakly $1$-generic.

However, the symmetric difference of $X$ and $Y$ is a subset
of $I^X$. Hence $I^X$ cannot be indifferent for $X$ with
respect to weak $1$-genericity. This contradiction completes
the proof of the theorem.
\end{proof}

\newpage
\section{Turing degrees of computably enumerable $\w$-c.a.-tracing sets}

D.\ Diamondstone and A.\ Nies started discussions in  December 2011. This work now  includes Joe Zheng. For background on tracing see~\cite[Sections 8.2,8.4]{Nies:book}.

\subsection{Introduction}   The
following (somewhat weak) highness property was introduced by
Greenberg and Nies~\cite{Greenberg.Nies:11}; it coincides  with the
class $\mathcal G$ in~\cite[Proof of 8.5.17]{Nies:book}.
\begin{definition} \label{def:omcetr} {\rm A set $A$ is \emph{$\omega$-c.a.-tracing} if each function $f \lwtt \Halt$ has a $A$-c.e.\ trace $(T_x^A)\sN x$ such that $|T_x^A| \le 2^x$ for each $x$.} \end{definition}

  A   stronger condition is that every $\DII$ function $f$ must be traced:

  \begin{definition} \label{def:Deltacetr} {\rm A set $A$ is \emph{$\DII$-tracing} if each $\DII$ function $f$ has a $A$-c.e.\ trace $(T_x^A)\sN x$ such that $|T_x^A| \le 2^x$ for each $x$.} \end{definition}  
  One also says that $\ES'$ is  c.e.\ traceable \emph{by} $A$.

\subsubsection{The Tewijn-Zambella argument}  \label{sss:TZ}

By an argument of Terwijn and
Zambella (see~\cite[Thm.\ 8.2.3]{Nies:book}),  in both cases the bound $2^x$ can be
replaced by any order function without changing the tracing property. The two condition  on a  class $\+ C$ of total functions being traced 
 which makes this argument work is   the following.

\bi \item[($*$)] \n {\it  If $f \in \+ C$ and $r$ is a  computable function,   then the function $x \to f \uhr{ r(x)}$ (tuples suitably encoded by numbers) is also in $\+ C$. } \ei

\subsubsection{Double highness properties} 
 These two classes relate nicely to highness for pairs of randomness notions by results in~\cite{Figueira.Hirschfeldt.ea:nd, Barmpalias.Miller.ea:nd}.   Recall that for randomness notions $\+ C \supset \+ D$, we let $\High(\+ C, \+ D)$ be  the class of oracles $A$ such  $\+ C^A \sub \+ D$ ($A$ is strong enough to push $\+ C$ inside $\+ D$). The following is  obtained by combining~\cite{Figueira.Hirschfeldt.ea:12,Barmpalias.Miller.ea:nd}:

\begin{thm}  \label{thm:highness_wcatraceable}  Let $A$ be an oracle.

\bi \item[(a)]  $A\in\High(\MLR, \Dem)$ $\LR $    $A$ is $\omega$-c.a.\ tracing.

\item[(b)]  $A\in\High(\MLR, \SR[\Halt])$  $\LR$ $A$  is $\DII$  tracing. \ei \end{thm}

  \subsubsection{Extension to randomness notions in between the two extremes}

It is known that  $\SR[\Halt])$ is the same as limit random (the set has to pass all Demuth-like tests where the number of version changes is merely finite; see \cite{Kucera.Nies:11} for the formal  definition). Let $\aaa$ be a computable limit ordinal and recall the definition of an $\aaa$-c.a.\ function;  $\aaa-\Dem$ is Demuth randomness extended to tests with    versions of   test components $\aaa$-c.a. That is, the tests have the form $\Opcl {W_{f(m)}}\sN m$ where $f$ is $\aaa$-c.a. This is  studied,  for instance,  in the last sections of  
\cite{Greenberg.Hirschfeldt.ea:nd}, and~\cite{Nies:11}.   The table  suggests an   extension to notions in between the two extremes, both on the randomness and the tracing side: 

\begin{conjecture}
Let $\aaa$ be a computable limit ordinal possibly with some additional  closure properties such as closure under $+$. Then \bc $A\in\High(\MLR, \aaa-\Dem)$ iff  $A$ is $\aaa$-c.a.\ tracing. \ec 
\end{conjecture}
  In full this is only known if the class of $\aaa$-c.a.\ function satisfies the condition ($*$)   in \ref{sss:TZ}; for instance, this is the case when $\aaa = \omega^n$ for some $n \in \omega$.  See  Subsection~\ref{ss:noTZ} for more detail.

  We say $A$ is (weak) Demuth cuppable if there is a (weak) Demuth random set $Z$ such that $A \oplus Z \ge_T \Halt$. Diamondstone and Nies also noticed that for each  c.e.\ set $A$, 
  
 \bc  $K$-trivial  $\RA$  weak Demuth noncuppable  $\RA$   superlow,    \ec

  For the first implication, $K$-trivial even implies ML-noncuppable by a recent result of  Day/Miller \cite{Day.Miller:nd}. The second one is shown in Figueira et al.~\cite{Figueira.Hirschfeldt.ea:12}. They use that for c.e.\ sets $A$, superlow = $\w$-c.e.\ jump dominated in the sense of~\cite{Figueira.Hirschfeldt.ea:12}; the negation of this second property implies $\High(\MLR, \text{weak Demuth})$. Thus if a c.e.\ set $A$ is not superlow, $\Om^A$ is weakly Demuth random and cups $A$ above $\Halt$. 
  
Demuth traceability was introduced in  \cite{Bienvenu.Downey.ea:nd}.  For instance,  all  superlow c.e.\ sets   are Demuth traceable.  They also observed that  for \emph{any}  set, 
  \bc Demuth traceable  $\RA$ Demuth noncuppable  $\RA$ not  $\w$-c.a.\ tracing.  \ec
  
 For the first implication see \cite{Bienvenu.Downey.ea:nd}. The second implication follows from Theorem~\ref{thm:highness_wcatraceable} above by taking $\Om^A$.

  \begin{question} Characterize weak Demuth noncuppability, and Demuth noncuppability, in recursion theoretic terms.  \end{question}

  \subsection{Can an $\w$-c.a.\ tracing set be  close to computable?}
  
  Being $\w$-c.a.\ tracing was investigated in   \cite{Figueira.Hirschfeldt.ea:nd}.  They showed:
  \begin{fact} \label{suplowNOTomegac.e.tracing}  No superlow set is $\omega$-c.e.-tracing.\end{fact}
  
  On the other hand, lowness is possible.  By \cite[Journal version Cor. 2.5]{Figueira.Hirschfeldt.ea:nd} there is an $\omega$-c.e.-tracing low ML-random set. Here we build a c.e.\ such set.
  
  \begin{prop}  \label{prop: low_om_ca_trace} Some low c.e.\ set $A$ is $\w$-c.a.\ tracing.  \end{prop}
  
  \begin{proof} (Idea) We obtain   $(f_e) \sN e$  a list of all $\w$-c.a.\ functions as follows.   Let $\la e\ra $ be the $e$-th wtt reduction procedure (namely, $e_0$ indicates a Turing functional and $e_1$ is a computable bound on the use).  At stage $s$ we have an approximation $f_e(x)[s] = \la e \ra^\Halt(x)[s] $, with value $0$ if this is undefined. The function $f_e$ is given by $f_e(x) = \lim_s \la e \ra^\Halt(x)[s]$. 
	         
	 We build a c.e. oracle trace $(T^Z_x) \sN x$ with a fixed computable bound $h(x)$. We  meet  \bc $P_{e,x} \colon   f_e(x) \in T_x^A$    $(e \le x)$.  \ec 
	We also meet the usual lowness requirements $L_i$. 
	
	Fix an effective  priority ordering of the requirements.
	
	\vsp

\n 	\emph{Strategy for $P_{e,x}$ at stage $s$:} when there is a new value $y= f_e(x)[s]$,  change $A$  to remove the previous value (if any), unless its $A$-use is  restrained by a  stronger priority $L$-type  requirement. Put $y \in T^A_x$ with large use on $A$.
	
	\vsps
	
\n 	\emph{Strategy for $L_i$ at stage $s$:} if $J^A(i)$ converges newly (by convention with use $\le s$), restrain all weaker $P$ requirements from changing  $A\uhr s$. 
	
	\begin{claim} There is a computable bound for $|T_x^A|$. \end{claim}

 	\end{proof}
We give  another formal explication of the idea that a c.e.\ set $A$ can be high in one sense and low in another.   This answers a question asked  in~\cite[Rmk.\ 31]{Figueira.Hirschfeldt.ea:12}  because for c.e.\ sets,  array rec.\ = c.e.\ traceable.   For detail see Zheng's master thesis. 

\begin{thm} Some  $\w$-c.a.\ tracing c.e.\ set $A$   is   c.e.\ traceable. \end{thm}

	Note that  the class of c.e.\ traceable   sets contains all  the superlow c.e.\ sets, but not all low c.e.\ sets. Thus the two results are independent.
	
	Also note that $A$ cannot be $\DII$ tracing: by Barmpalias~\cite{Barmpalias:AC}, such a set is not weakly arrary recursive (for each $\DII$ function $g$ there is $h \leT A$ such that $\ex^\infty x h(x) > g(x)$), hence not c.e.\ traceable. 
	
\begin{proof} Use   a  $\ES''$ tree construction. To make $A$ c.e.\ traceable we meet requirements 
	\bc $S_i\colon \, \Phi_i^A $ total $ \RA $  build c.e. trace  $ (V_x)\sN x$ for $\Phi_i^A$, \ec
	
	where $|V_x|$ has fixed computable bound. Guess at $\Phi_i^A $ total on the tree via exp stages; $A$ becomes $low_2$. 
	
	 To make $A$   $\w$-c.a.\ tracing, meet requirements $P_{e,x}$ as above. Use same strategies as  above but  on the tree. A sty $\beta\colon P_{e,x}$ below the infinitary outcome of  a strategy $\aaa\colon S_i$ can only start once   $\Phi_i^A \uhr y$ has converged for large enough $y$, so that its $A$ changes don't make $|V_y|$ to large. For this we need the advance bound on the number of times $f_e(x)$ can change. \end{proof}

\subsection{Can a non-$\w$-c.a.\ tracing set be  close to $\Halt$?}

For any   $\DII$  set $A$, superhighness is equivalent to JT-hardness by \cite[8.4.27]{Nies:book}, which of course implies $\w$-c.a.\  tracing. Thus, every superhigh $\DII$ set is $\w$-c.a.\ tracing. The following is pretty sharp then.

\begin{thm} Some high c.e.\ set $A$ is not $\w$-c.a.\ tracing. \end{thm}
\begin{proof}  Define an $\w$-c.a.\ function $g$.  We meet the requirements
	
	\bc $N_{e}\colon \ex x \, g(x) \not \in T_{e,x}^A$, \ec
	where $(T^Z_{e,x})_{e,x \in \w}$ is uniform  listing of all oracle c.e.\ traces with bound $x$. 

	We also meet the usual highness requirements 
	
	\bc $P_r\colon     \ES''(r)  = \lim_n A^{[r]}(n) $, \ec by the usual coding into the $r$-th column of $A$.
	
Use methods from the tree construction of a high minimal pair as in \cite{Soare:87}. Guess on the tree whether $ \ES''(r)=0$.  This yields  a notion of $\aaa$-correct $A$-computations, where $\aaa$ is a string.

\emph{	Strategy for $\aaa \colon N_e$.} 
	Pick large $x$. Define $g(x)$ large. Whenever $g(x)\in T_{e,x}^A$ via an $\aaa$-correct computation, then increase $g(x)$ and initialize weaker   requirements.

\n \emph{Verification.} 
Need to check that $g$ is indeed $\w$-c.a. As long as  $\aaa \colon N_e$ is not initialized, all relevant computations $y\in T_{e,x}^A$ are $A$-correct, so it will  increase $g(x)$ at most $x$ times.
  \end{proof} 
  
  For an alternative, full proof,  again see Zheng's master thesis.  There, the $
\alpha$ correct computations are replaced by a more informative tree of strategies.

\subsection{Traceability in absence of the condition of Terwijn and Zambella}
\label{ss:noTZ}

We generalize  Definitions \ref{def:omcetr} and \ref{def:Deltacetr}.
\begin{definition} \label{def:fullcetr} {\rm Let $\+ C$ be a class of total functions defined on $\NN$. Let $h \colon \NN \to \NN$. We say that  a set $A$ is \emph{$\+ C$-tracing with bound $h$ } if each function $f \in \+ C$ has a $A$-c.e.\ trace $(T_x^A)\sN x$ such that $|T_x^A| \le h(x)$ for each $x$.} \end{definition}

Note that by the Terwijn Zambella argument, the order function $h$ is immaterial if Condition ($*$) of Subsection~\ref{sss:TZ} holds for $\+ C$.

\begin{definition}\label{def:C-Demuth}
	 Let $\+ C$ be a class of total functions defined on $\NN$. \bi \item A \emph{$\+ C$-Demuth  test}  has  the form $\Opcl {W_{f(m)}}\sN m$ where $f$ is in $\+ C$. \item  We say that $Z$ is \emph{$\+ C$-Demuth random} if for each such test we have $Z \not \in \Opcl {W_{f(m)}} $ for almost all $m$. \ei
\end{definition}

We give a fine analysis  of Theorem~\ref{thm:highness_wcatraceable} in this more general setting. We show the double highness notion $A\in\High(\MLR, \+ C-\Dem)$  implies $A$ is $\+ C$-tracing at bound $2^m$. However, we need tracing at slightly better  bound, such as $2^m m^{-2}$, to reobtain the double highness notion.  Thus there is a   gap in the absence of condition ($*$).

The first part is a  straightforward modification of the proof of   \cite[Thm. 3.6]{Barmpalias.Miller.ea:nd}.
\begin{lemma}\label{lem:random_to_tracing}  
	$A\in\High(\MLR, \+ C-\Dem)$ $\RA $   
	
	\hfill  $A$ is $\+ C$-tracing with bound $2^n$. 
\end{lemma}

The second part is a    modification of the proof of the corresponding result  \cite[Prop.\ 32]{Figueira.Hirschfeldt.ea:nd}; also see   \cite[Thm. 3.6]{Barmpalias.Miller.ea:nd}.
\begin{lemma}\label{lem:tracing_to_random}  
 $A$ is $\+ C$ tracing with a bound $g$  such that $\sum g(n) \tp{-n} < \infty$ $\RA $     
  	$A\in\High(\MLR, \+ C-\Dem)$. 
\end{lemma}
\begin{proof} 
	Fix a $\+ C$ Demuth test $\Opcl {W_{f(m)}}\sN m$. Let $(T_m)\sN m $ be a $A$-c.e.\  trace for $f$ with bound $g$. Define an $A$ Solovay test $(\+ S^A_m)$ as follows: for each $k \in T_m$, enumerate the open set $\Opcl {W_{f(k)}}$ into $\+ S^A_m$ as long as its  measure is $\le \tp{-k}$. Clearly $\sum_m \leb \+ S^A_m \le \sum_m g(m) \tp{-m}  < \infty$. Thus no set that is in infinitely many $\+ S^A_m$ can be ML-random relative to $A$. 
\end{proof}

We   now look at  subclasses of the $\omega$-c.a.\ functions, and also classes $\+ C $ containing the $\omega$-c.a.\ functions.  

 
\subsubsection{Subclasses of $\omega$-c.a.}   For a computable order function $g$, let $\+ C$ be the class of functions $f \lwtt \Halt$ such that   some computable  approximation  for  $f(x)$ has at most  $g(x) $ changes. Then  $A$ is $\+ C$ tracing  with bound $2^m$ if each function in $\+ C$    has an $A$-c.e.\ trace $(T_x^A)\sN x$ such that $|T_x^A| \le 2^x$ for each $x$. We also say that $A$ is \emph{$g$-c.a.-tracing} with bound $2^m$.   More generally, we could have a class $\+ D$ of computable functions instead of a single $g$, and we say $A$ is $\+D$-c.a.\ tracing with the obvious meaning. 

Let $Z$ be ML-random. 
By~\cite[Thm 23]{Figueira.Hirschfeldt.ea:12}, $Z$ is $\w$-c.a.\ tracing iff $Z$ is $2^n h(n)$-c.a.\ tracing, where $h$ is an arbitrary (say, slowly growing) order function.  

In contrast,  for c.e.\ sets $A$, there is a proper hierarchy of being $g$-c.a.\ tracing, for faster and faster growing computable functions $g$:

\begin{theorem}\label{thm:}   Let $g$ be  computable. Then there is a computable function $h$ and a c.e.\ set $A$ that 
	is $g$-c.a.\ tracing, but not $h$-c.a.\ tracing.
\end{theorem}

\begin{proof}[Sketch of proof] Modifying the proof of Proposition~\ref{prop: low_om_ca_trace}, we can build a   $g$-c.a.\ tracing c.e.\ set $A$ that is superlow. Hence this set is not $h$-c.a.\ tracing for an appropriate faster growing $h$.
\end{proof}

We give an example where there is no gap. Logarithms are in  base 2.
%
%
%
%
%
%
%
%
%

\begin{proposition}\label{pro:alpha n}
	Let $\+ D$ be a class of computable  bounds on the number of changes such that   for each   function $r \in \+D$, the function  \bc $  \sum_{k \le n + 3 \log  n} r(k)$  \ec is also in $\+ D$. Let $\+C $ be the class of  $\+ D$-c.a.\ functions. Then  
	\bc $A\in\High(\MLR, \+ C-\Dem)$  $\LR $  $A$ is $\+ C$-tracing with bound $\tp n$. \ec
\end{proposition}

An example of such a $\+ D$ is the class of computable  functions bounded by   a function of type $2^n s(n)$ where $s$ is a polynomial.

\begin{proof}[Sketch of proof]    $\RA:$ by Lemma~\ref{lem:random_to_tracing}.

\n $\LA:$  By Lemma \ref{lem:tracing_to_random}  it suffices to show that $A$ is $\+ C$ tracing with a better bound $g$, namely,  $g$ is a function such that  $\sum g(n) \tp{-n} < \infty$. 

Let   $q(n) = \max \{ i \colon \, i  + 3 \log  i \le n \}$ and let $g(n) = \tp{q(n)+1}$.  Note  that by hypothesis on $\+ D$, for each $f \in \+ C$ the function $ \hat f (i) = f \uhr {i  + 3 \log  i}$ is also in $\+ C$. Since we can determine $f(n) $ from $\hat f(q(n)+1)$,  the value $  f(n)$ can be traced  by $A$ with bound $\tp {q(n)+1}$ as in  the Terwijn Zambella argument.    

 It is easy to check that for a.e.\ $n$, 
$$q(n) \le r(n) := n - 2 \log n,$$
using that $n - 2 \log n + 3 (n -  2 \log n) > n$.
Since   $\tp{r(n) -n } =  n^{-2}$, we have   $\sum g(n) \tp{-n} < \infty$ as required.

%
%
%
\end{proof}
\subsubsection{Classes $\+ C $ containing $\omega$-c.a.}

\begin{proposition}\label{pro:alpha n}
	The class of $\+ C$ of $\omega^n$-c.a.\ functions satisfies the  condition $(*)$. Hence  $A\in\High(\MLR, \+ C-\Dem)$  $\LR $  $A$ is $\+ C$-tracing.
\end{proposition}

\begin{proof}[Sketch of proof]  One modifies \cite[Lemma 5.2]{Nies:11} where $n=2$.  This can be expanded suitably to $\omega^n$. See Zhengs thesis.
\end{proof}

\newpage

\part{Higher randomness}

\newcommand{\QI}{\Pi^1_1}
\newcommand{\TI}{\Sigma^1_1}
\newcommand{\DI}{\Delta^1_1}
\newcommand{\fh}{\text{\rm fin-h}}
\newcommand{\OCK}{\omega_1^{CK}}

\section{Notions stronger than $\QI$-ML-randomness  \\
 (Chong, Nies and Yu)}

Chong, Nies and Yu worked in Singapore in April. They considered analogs or randomness notions stronger than \ML's  in the realm of effective descriptive set theory. For background see \cite[Ch.\ 9]{Nies:book}. As there, the higher analog of a concept is generally  obtained by replacing ``c.e.'' with ``$\QI$ '' everywhere. The notation for a higher analog is obtained by underlining the previous notation. For instance, $\ul \Om$ is the higher analog of Chaitin's $\Om$, and $\ul \MLR$ denotes $\QI$-\ML\ randomness.

The following can be seen as the higher analog of Turing reducibility. The number of stages of an oracle  computation is a computable ordinal, yet the use  on the oracle is finite.
\begin{definition}[Hjorth and Nies \cite{Hjorth.Nies:07}] \label{uc:finh}
 A {\it $\fh$ reduction procedure}   is a partial
function $\Phi\colon \, \strcantor\rightarrow \strbaire$  with~$\QI$ graph
such that $\dom(\Phi)$ is closed under prefixes
and, if $\Phi(x) \downarrow$ and  $ y \preceq x$, then  $ \Phi(y) \preceq \Phi(x)$.
We write  $f = \Phi^Z$ if $\fa n \ex m  \ \Phi(Z\uhr m)
\succeq f\uhr n$, and $f \le_{\fh} Z$ if $f = \Phi^Z$ for some $\fh$ reduction procedure~$\Phi$.
\end{definition}

Bienvenu, Greenberg and Monin (July) have observed that this isn't the right notion for most purposes. One shouldn't have the closure under prefixes in most results below.  They call the more general version $\le_{hT}$ (higher Turing).

They may add a summary  of their work here, which will appear independently.

 \begin{definition}[Bienvenu, Greenberg and Monin ] \label{hT}
 A {\it $hT$ reduction procedure}   is a partial
function $\Phi\colon \, \strcantor\rightarrow \strbaire$  with~$\QI$ graph
such that  if $\Phi(x) \downarrow$ and  $ y \preceq x$, then  $ \Phi(y) \preceq \Phi(x)$.
We write  $f = \Phi^Z$ if $\fa n \ex m  \ \Phi(Z\uhr m)
\succeq f\uhr n$, and $f \le_{hT} Z$ if $f = \Phi^Z$ for some $hT$ reduction procedure~$\Phi$.
\end{definition}

\subsection{Higher analog of weak 2-randomness}

The following was introduced in \cite[Problem 9.2.17]{Nies:book}.

\begin{df}  A  \emph{generalized~$\QI$-ML-test}
 is a  sequence $(G_m)\sN{m}$ of
uniformly~$\QI$ open classes such that  $\bigcap_m G_m$ is a null
class. $Z$ is  \emph{$\QI$-weakly 2-random}  if
$Z$ passes each generalized~$\QI$-ML-test. \end{df}
Clearly the class $\bigcap_m G_m$ is $\QI$, so $\QI$-randomness implies  $\QI$-weak 2 randomness.

\cite[Problem 9.2.17]{Nies:book}  asked   whether the new notion coincides with $\QI$-ML-randomness.
\begin{theorem}
If $x$ is the leftmost path of a $\Sigma^1_1$-closed set of reals, then $x$ is not $\QI$-weak 2 random.
\end{theorem}
\begin{proof}

Let $T\subseteq 2^{<\omega}$ be a $\Sigma^1_1$-tree.

For any $n\in \omega$ and $\alpha<\CK$, let $$U_{n,\alpha}=\{\sigma\mid \exists z (z\mbox{ is the leftmost path in }T[\alpha] \wedge \sigma=z\uh n+1)\}.$$

Define $$U_{n,<\alpha}=\bigcup_{\beta<\alpha}U_{\beta}$$ and $$U_n=\bigcup_{\alpha<\CK}U_{n,\alpha}.$$

The following facts are obvious.
\begin{enumerate}
\item For any $n$ and $\alpha<\CK$,  $U_{n+1,\alpha}\subseteq U_{n,\alpha}$;
\item For any $n$ and $\alpha<\CK$, $\mu(U_{n,\alpha})<2^{-n}$;
\item $x\in \bigcap_{n\in \omega}U_n$;
\item For any $n$, $\alpha<\CK$ and real $z$, if $z\in U_{n,<\alpha}\setminus U_{n,\alpha}$, then $z\not\in U_{n,\beta}$ for any $\beta\geq \alpha$.
\end{enumerate}

Now suppose that $\mu(\bigcap_{n\in \omega}U_{n})>0$, then there must be some $\sigma_0$ so that $$\mu( \bigcap_{n\in \omega}U_{n} \cap [\sigma_0])>\frac{3}{4}\cdot 2^{-|\sigma_0|}.$$ Let $n_0=|\sigma_0|+2$. Then there must be some least $\alpha_0<\CK$ so that
$$\mu( U_{n_0,\leq \alpha_0} \cap [\sigma_0])>\frac{3}{4}\cdot 2^{-|\sigma_0|}.$$

By (2), $$\mu( U_{n_0,< \alpha_0} \cap [\sigma_0])>\frac{1}{2}\cdot 2^{-|\sigma_0|}.$$

By (1) and (4), $$\mu( \bigcap_{n>n_0}U_{n,<\alpha_0} \cap [\sigma_0])>\frac{1}{4}\cdot 2^{-|\sigma_0|}.$$

So there must be some $\sigma_1\succ \sigma_0$ so that $$\mu( \bigcap_{n>n_0}U_{n,<\alpha_0} \cap [\sigma_1])>\frac{3}{4}\cdot 2^{-|\sigma_1|}.$$

Let $n_1=|\sigma_1|+2$. Then there must be some least $\alpha_1<\alpha_0$ so that
$$\mu( U_{n_1,\leq \alpha_1} \cap [\sigma_1])>\frac{3}{4}\cdot 2^{-|\sigma_1|}.$$

Repeat the same method, we obtained a descending sequence $\alpha_0>\alpha_1>\cdots$, which is a contradiction.
\end{proof}

In the computability setting, let $Z$ be ML-random. Then  by   a result of Hirschfeldt and Miller (see \cite[5.3.16]{Nies:book}),
\bc $Z$ is weakly 2-random $\LR$ $Z, \Halt$ form a Turing minimal pair.   \ec
We cannot expect this to hold   in the higher setting, because by Gandy's basis theorem, there is a $\QI$-random set $Z \leT \+ O$.  However, the result of Hirschfeldt and Miller actually shows that if  $Z$ is not weakly 2-random, then  there is a c.e.\ incomputable set $A$ below $Z$. This carries over and yields a characterization of the $\QI$-weakly 2-random sets withing the ML-random sets.
%
%
%
%

\begin{thm} \label{thm: char W2R} Consider the following  for a $\QI$-ML-random set $Z$.

\bi
\item[(i)]  $Z$  is $\QI$-weakly 2-random

\item[(ii)] If $A \le_{hT} Z$ for a $\QI$ set $A$,   then $A$ is hyperarithmetical.
\ei

We have that  (ii)$\rightarrow$(i).
\end{thm}

\begin{proof}

 We can adapt some of the theory of cost functions to the higher setting.
 Suppose $Z \in \bigcap_m G_m$ where $(G_m) \sN m$ is a generalized~$\QI$-ML-test. Let $c(x, \aaa) = \leb G_{x, \aaa}$. This is a higher cost function with the limit condition. Hence there is a $\QI$ but not $\DI$ set $A$ obeying~$c$ by the higher analog of \cite[5.3.5]{Nies:book}. By the higher analog of  the Hirschfeldt-Miller  method in the version \cite[5.3.15]{Nies:book}, we may conclude that $A \le_\fh Z$.

 (Maybe you, the reader, are puzzled why this doesn't show  that the   $\QI$ random set  $Z \leT \+ O$ obtained by Gandy's basis theorem we discussed earlier $\fh$ bounds a properly $\QI$ set? The answer is that, unlike the  case of $\DII$ sets in  the computability setting, not for every set $Z \leT \+O$  there is a generalized~$\QI$-ML-test $(G_m)$ with $\bigcap G_m = \{Z\}$. )

 \end{proof}

%
%
%
%
%
%
%
%

We now attempt a partial solution to \cite[Problem 9.2.17]{Nies:book}, showing that as operators sending oracles to classes, the two randomness notions differ.

\begin{conjtheorem}  $\QI$ weak 2-randomness and $\QI$ randomness differ relative to some low oracle. \end{conjtheorem}

\begin{proof} Let $\ul \Om_0$ be the bits of $\ul \Om$ in the even positions, and let $\ul \Om_1$ be the bits of $\ul \Om$ in the odd positions. Clearly $\ul \Om_0$, $\ul \Om_1$ are both $\QI$-\ML\ random and $\fh$ incomparable.
By the van Lambalgen theorem for $\QI$-randomness \cite{Hjorth.Nies:07}, if $\ul \Om_0$ is $\QI$ random, then $\ul \Om_1$ is not $\QI$ random relative to $\ul \Om_0$. It now suffice to show that $\ul \Om_1$ is  $\QI$ weakly 2-random relative to $\ul \Om_0$. Since it is $\QI$ ML-random relative to  $\ul \Om_0$, by the theorem~\ref{thm: char W2R} relative to $\ul \Om_0$ this amounts to  showing that $\+ O ^{\ul \Om_0} \not \le_\fh \ul \Om$.
(Can someone show this? Seems to be analogous to the fact \cite{Nies:book}  3.4.15) that halves of $\Om$ are not superlow. ).
 \end{proof}

\subsection{The higher analog  of difference randomness}

 \begin{df}  A  \emph{$\QI$-difference-test}
 is a  sequence $(G_m)\sN{m}$ of
uniformly~$\QI$ open classes together with a  closed $\TI$ class $\+ C$ such that  $\leb (\+ C \cap  G_m) \le \tp{-m}$ for each $m$. We say that  $Z$ is  \emph{$\QI$-difference random}  if
$Z$ passes each such test in that $Z \not \in \+ C \cap \bigcap_m G_m$. \end{df}

 \begin{thm} \label{thm: char DiffR} Let $Z$ be a $\QI$-ML-random set. Then
\bc  $Z$ is $\QI$-difference random $\LR$ $\ul \Om  \not \le_\fh Z$. \ec
\end{thm}

Note: again, $\le_\fh$ has to be replaced by $\le_{hT}$ to make this work (Bienvenu, Greenberg and Monin).
 \begin{proof}

 \lapf By contraposition. Suppose that $Z  \in \+ C \cap \bigcap_m G_m$ for a $\QI$-difference test $\+ C, (G_m)\sN{m}$.
 Define a higher Solovay test  $\+S$ as follows. When $m$ enters $\+O $ at stage $\beta$, we may within  $L_{\OCK}$ compute $\gamma \ge \beta$ such that  $\tp{-m} \ge  \leb   (\+ C_\gamma  \cap G_{m, \beta})$.  (More precisely, the function $h \colon \OCK \times \omega \to \OCK$ mapping $\beta, m$  to $\gamma$ if $m$ enters at stage $\beta$, and to $0$ otherwise, is    $\SI 1$ over $L_{\OCK}$.) By \cite[1.8.IV]{Sacks:90}, determine an open  $\DI$ class  $\+ D_m \supseteq (\+ C_\gamma  \cap G_{m, \beta})$ such that $\leb \+ D_m  \le \tp{-m+1}$, and enumerate $\+ D_m$  into $\+ S$.

 Let
  $g(m) = \leb \aaa. [ Z \in  G_{m, \aaa}]   $.
 Note that $g \le_\fh Z$. Since $Z$ is $\QI$ ML-random, we have $Z \not \in \+ D_m$ for almost every $m$. Hence we have $m \in \+ O  \lra m \in \+ O_{g(m)}$ for a.e.\ m, which shows that $\+ O \le_{\fh} Z$.

  \rapf  Suppose that $\ul \Om = \Gamma^Z$ for a $\fh$ reduction procedure $\Gamma$.
  Choose  $c\in \NN$ such that $\tp{-n} \ge \leb \{ Y \colon \ul \Om \uhr{n+c} \preceq \Gamma^Y\}$. Let
   \[ V_{n, = \aaa} = \Opcl { \{\sss\colon\,   [ \Gamma_\aaa^\sss \succ \ul \Om_\aaa \uhr {n+c} \lland \leb  \{ \tau \colon \, \Gamma_\aaa^\tau \succ \ul \Om_\aaa \uhr {n+c}\} \le \tp{-k}] \}}. \]
   Let $V_{n, \gamma} = \bigcup_{\aaa \le \gamma} V_{n, = \aaa}$, and
   $\+D = \bigcup_n \bigcup_\aaa \{ V_{n, =\aaa}\colon \, \ex \beta > \aaa [ \ul \Om_\aaa < \ul \Om_\beta]$.
   Note that $\+ D$ is $\QI$ open, and $V_n$ is $\QI$ open uniformly in $n$.  By definition we have $\tp{-n} \ge \leb V_n \setminus \+ D$.  Clearly $Z \in \bigcap_n V_n  \setminus \+ D$. Thus $Z$ is not higher difference random.
\end{proof}

 \subsection{Randomness and density}

  \begin{definition} {\rm
 We define the (lower) \emph{Lebesgue density}  of a set $\+ C \subseteq \RR$ at a point~$x$ to be the quantity $$\rho(x|\+ C):=\liminf_{\gamma,\delta \rightarrow 0^+} \frac{\lambda([x-\gamma,x+\delta] \cap \+ C)}{\lambda([x-\gamma,x+\delta])}.$$
}

For $x \in \RR$ and $m \in \w$ we denote by $[x\uhr m)$ the interval of the form $[k \tp{-m}, (k+1)\tp{-m})$ containing    $x$.
The \emph{dyadic   density}  of a set $\mathcal C \subseteq \RR
$ at a point~$x$  is

 $$\rho_2(x|\mathcal C):=\liminf_{n \rightarrow \infty} \frac{\lambda([x\uhr n ) \cap \mathcal C)}{\lambda([x\uhr n))}.$$
  \end{definition}

The following result it the higher  analog of a result on difference randomness due to \cite{Bienvenu.Hoelzl.ea:12a}. Only the usual notational changes to the proofs are necessary, which we omit.

\begin{theorem}\label{thm:density_Turing}
Let $x$ be a $\QI$ Martin-L\"of random real. Then $x\not \geq_{\fh} \+ O$ iff $x$ has positive density in  some closed $\TI$ class containing $x$.
\end{theorem}

As in \cite{Bienvenu.Hoelzl.ea:12a}, the result is derived from a lemma which actually shows that the $\TI$ class is the same on both sides.
\begin{lemma}\label{lem:density}
Let $x$ be a $\QI$  Martin-L\"of random real. Let $\+C$ be a closed $\TI$ class  containing $x$. The following are equivalent:
\begin{enumerate}
\item[(i)] $x$ fails a $\QI$ difference test of the form $((U_n)\sN n,\+C)$.
\item[(ii)] $x$ has lower Lebesgue density zero in $\+C$, i.e., $\rho(x|\+C)=0$.
\end{enumerate}
\end{lemma}

 The following   was first  discussed in April in  Paris (Bienvenu, Monin and Nies). The proof below is due to   Yu. It does the weaker case of dyadic density, but could be adapted to full Lebesgue density.

 \begin{prop} Let $x \in \cantor$ be $\QI$-random. Suppose $ x \in \+ C$ for a $\TI$ class $\+ C$. Then
  $\underline{\lim}_{n\to\infty} \leb_{[x\uh n]}(\+ C) = 1$. \end{prop}

  \begin{proof} Suppose otherwise. Then
 for some rational $p<1$, there are infinitely many $n$ such that  $\leb_{[x\uh n]}(\+ C) < p$.  Define a  function $f$ that is $\Sigma_1$ over $L_{\omega_1^x}[x]$  as follows: for each $k$, $f(k)$ is the $<_{L[x]}$-least pair $(m_k,\alpha_k)$ so that $m_k>k$ and $ \leb_{[x\uh m_k]}\+C [\alpha_k]  ) <p$. Then $f$ is a total function. So there must be an ordinal $\gamma<\omega_1^x=\CK$ so that $ \leb_{[x\uh m_k]}\+C [\alpha_k]  <p$ for every $k$.  This implies $$\underline{\lim}_{n\to\infty} \leb_{[x\uh n]}\+C [\gamma]   <p.$$
  Since $\+C\subseteq \+C [\gamma]$,  we have $x\in \+C[\gamma]$. But $\+ C [\gamma]$ is a $\Delta^1_1$ set,   so $ x$ has density $1$ in $\+C [\gamma]$, a contradiction.
 \end{proof}

 \subsection{A higher version of  Demuth's Theorem}

 \begin{theorem}\label{theorem: higher demuth}
 If $x$ is $\Pi^1_1$-random and $y\leq_h x$ is not hyperarithmetic, then there is a $\Pi^1_1$ random real $z\equiv_h y$.
 \end{theorem}

\begin{proof}
Suppose that $x$ is $\Pi^1_1$-random and $y\leq_h x$ is not hyperaritmetic. Then there is a  there is some recursive ordinal $\alpha$, an nondcreasing function $f\leq_T \emptyset^{(\alpha)}$ and a recursive function $\Psi$ so that $\lim_{n\to \infty}f(n)=\infty$ and for every $n$, $$y(n)=\Psi^{x \uh f(n)\oplus \emptyset^{(\alpha)\uh f(n)}}(n)[f(n)].$$

For each $u\in 2^{<\omega}$, let
$$l(u)=\sum_{\tau \in 2^{|u|} \wedge \tau<u }(\sum \{2^{-|\sigma|} \colon \, \sigma\in 2^{f(|u|)}\wedge \Psi^{\sigma\oplus \emptyset^{(\alpha)}\uh f(|u|)}[f(|u|)]\uh |u|=\tau\})$$ and
 $$r(u)=l(u)+\sum\{ 2^{-|\sigma|}  \colon \, \sigma\in 2^{f(|u|)}\wedge \Psi^{\sigma\oplus \emptyset^{(\alpha)}\uh f(|u|)}[f(|u|)]\uh |u|=\tau\},$$  where $\tau<u$ means that  $\tau$ is in the left of $u$.

One may view $\sum_{\sigma\in 2^{f(|u|)}\wedge \Psi^{\sigma\oplus \emptyset^{(\alpha)}\uh f(|u|)}[f(|u|)]\uh |u|=\tau}2^{-|\sigma|}$ as a kind of  ``measure'' for $\tau$.

For each $n$, let $$l_n=l(y\uh n), \mbox{ and } r_n=r(y\uh n).$$ Then $l_n\leq l_{n+1}\leq r_{n+1}\leq r_n$ for every $n$.

Since $y$ is not hyperarithmetic, it is not difficult to see that $\lim_{n\to \infty}r_n=0$. So there is a unique real $$z=\bigcap_{n\in \omega}(l_n,r_n).$$

Obviously $z\leq_T y\oplus \emptyset^{(\alpha)}$.  We leave readers to check that $y\leq_T z\oplus \emptyset^{(\alpha)}.$ So $z\equiv_h y$.

Suppose that $z$ is not $\Delta^1_1$-random. Then there must be some recursive ordinal $\beta<\alpha$ and a $\emptyset^{(\beta)}$-ML-test $\{V_n\}_{n\in \omega}$ so that $z\in \bigcap_{n\in \omega}V_n$. Let \begin{multline*}\hat{V}_n=\{u\mid \exists \nu(\nu \mbox{ is the $k$-th string in }V_n \wedge \\ \exists p\in \mathbb{Q}\exists q\in \mathbb{Q}(l(u)\leq p<q \leq r(u)\wedge [p,q]\subseteq [\nu]\wedge q-p>r(u)-2^{-n-k-2})\}.\end{multline*}

Since $z\in V_n$, we have that $y\in \hat{V}_n$ for every $n$. Note that $\{\hat{V}_n\}_{n\in \omega}$ is  $\emptyset^{(\beta+1+\alpha)}$-r.e.

Let $$U_{n}=\{\sigma\mid \exists \tau\in \hat{V}_n(|\sigma|=f(|\tau|)\wedge \Phi^{\sigma\oplus \emptyset^{(\alpha)}\uh f(|\tau|)}[f(|\tau|)]\uh |\tau|=|\tau|)\}.$$

Then $\{U_n\}_{n\in \omega}$ is  $\emptyset^{(\beta+1+\alpha)}$-r.e and $x\in \bigcap_{n\in \omega}U_n$.  Note that for every $n$, $$\mu(U_n)\leq \mu(V_n)+\sum_{k\in\omega}2^{-n-k-2+1}<2^{-n}+2^{-n}=2^{-n+1}.$$ Then $\{U_{n+1}\}_{n\in \omega}$ is a $\emptyset^{(\beta+1+\alpha)}$-ML-test. So $x$ is not a $\Delta^1_1$-random, a contradiction.
\end{proof}

An immediate conclusion of the proof of Theorem \ref{theorem: higher demuth} is:
\begin{corollary}
For any $\Pi^1_1$-random real $z$,  if $x\leq_h z$ is not hyperarithmetic, then $x$ is $\Pi^1_1$-random relative to some measure $\mu$.
\end{corollary}
\subsection{Separating lowness for higher randomness notions}\
In \cite{Yu11}, Yu gave a new proof of the the following theorem.
\begin{theorem}[Martin and Friedman]\label{theorem: martin and friedman theorem}
For any $\Sigma^1_1$ tree $T_1$ which has uncountably many infinite paths, $[T_1]$ has a member of each hyperdegree greater than or equal to the hyperjump.
\end{theorem}

\begin{lemma}\label{lemma: simga11 joint lemma}
Given any two uncountable $\Sigma^1_1$ sets of reals $A_0$ and $A_1$, for any real $z\geq_h \KO$, there are   reals $x_0\in A_0$ and $x_1\in A_1$ so that $x_0 \oplus x_1\equiv_h z$.
\end{lemma}
\begin{proof}

Fix a real $z\geq_h \KO$, two uncountable $\Sigma^1_1$-sets $A_0$ and $A_1$. So there are two recursive trees $T_0,T_1\subseteq 2^{<\omega}\times \omega^{<\omega}$ so that $A_i=\{x\mid \exists f\forall n (x\uh n,f\uh n)\in T_i\}$ for each $i\leq 1$. We may assume that neither $A_0$ nor $A_1$ contains a hyperarithmetic real. We also fix a recursive tree $T_2\subseteq \omega^{<\omega}$ so that $[T_2]$ is uncountable but does not contain a hyperarithmetic infinite path. Let $f$ be the leftmost path in $T_2$. Then $f\equiv_h \KO$.

Given a finite string $\sigma\in 2^{<\omega}$, we say that $\sigma$ is {\em splitting} on a tree $T\subseteq 2^{<\omega}\times \omega^{<\omega}$ if for any $j\leq 1$, $T_{\sigma^{\smallfrown}j}=\{(\sigma',\tau')\mid (\sigma'\succeq\sigma^{\smallfrown}j \vee \sigma'\prec\sigma^{\smallfrown}j )\wedge (\sigma',\tau')\in  T\}$ contains an infinite path. For any $i\leq 1$ and $(\sigma,\tau)\in T_i$, define $$T_{i,(\sigma,\tau)}=\{(\sigma',\tau')\in T_i\mid (\sigma',\tau')\succeq (\sigma,\tau)\vee (\sigma',\tau')\prec(\sigma,\tau)\}.$$

We shall construct a sequence $(\sigma_{i,0},\tau_{i,0})\prec (\sigma_{i,1},\tau_{i,1})\prec \cdots$ from $T_i$ for each $i\leq 1$. The idea is to apply a mutually coding argument. In other words, we will  use $\sigma_{0,j}$ to code $z(j)$, $f(j)$ and $\tau_{1,j-1}$,  and $\sigma_{1,j}$ to code $\tau_{0,j}$ for each $j\in \omega$.

\bigskip
At stage $0$, let $(\sigma_{i,0},\tau_{i,0})=(\emptyset,\emptyset)$ for $i\leq 0$. Without loss of generality, we may assume that $(\emptyset,\emptyset)$ is a splitting node in both $T_0$ and $T_1$.

At stage $s+1$:

Substage $1$, let $\sigma^0_{0,s+1}$ be the left-most finite splitting string extending  $ \sigma_{0,s}^{\smallfrown}z(s)$ on $T_{0,(\sigma_{0,s},\tau_{0,s})}$. So we code $z(s)$ here. Then we prepare to code $\tau_{1,s}$. Let $n^0_{s+1}=|\tau_{1,s}|-|\tau_{1,s-1}|$. Inductively, for any $k\in [1, n^0_{s+1}]$, let $\sigma_{0,s+1}^k$  be the left-most finite splitting string extending  $ (\sigma^{k-1}_{0,s+1})^{\smallfrown}1$ on $T_{0,(\sigma_{0,s},\tau_{0,s})}$ so that there are $\tau_{0,s+1}(k+|\tau_{0,s}|)$-many splitting nodes between $\sigma^{k-1}_{0,s+1}$ and $\sigma^k_{0,s+1}$.   Let $\sigma_{0,s+1}^{n^0_{s+1}+1}$  be the left-most finite splitting string extending  $ (\sigma^{n^0_{s+1}}_{0,s+1})^{\smallfrown}1$ on $T_{0,(\sigma_{0,s},\tau_{0,s})}$ so that there are $f(s)$-many splitting nodes between$\sigma_{0,s+1}^{n^0_{s}+1}$ and $\sigma_{0,s+1}^{n^0_{s+1}+1}$. So we code $f(s)$ here. Inductively, for $j\leq 1$, let $\sigma_{0,s+1}^{n^0_{s+1}+1+j+1}$ be the next splitting string in $T_{0,(\sigma_{0,s},\tau_{0,s})}$ extending $(\sigma_{0,s+1}^{n^0_{s+1}+1+j})^{\smallfrown}1$. This coding tells us that the action at this stage for $ (\sigma_{0,s+1},\tau_{0,s+1})$-part is finished. Define $\sigma_{0,s+1}=\sigma_{0,s+1}^{n^0_{s+1}+3}$. Let $\tau_{0,s+1}\in \omega^{|\sigma_{0,s+1}|}$ be the leftmost finite string so that the tree $T_{0, (\sigma_{0,s+1},\tau_{0,s+1})}$ has an infinite path.

Substage $2$. Let $\sigma_{1,s+1}^{0}=\sigma_{1,s}$ and $n^1_{s+1}=|\tau_{0,s+1}|-|\tau_{0,s}|$. Inductively, for any $k\in [1, n^1_{s+1}]$, let $\sigma_{1,s+1}^k$  be the left-most finite splitting string extending  $ (\sigma^{k-1}_{1,s+1})^{\smallfrown}1$ on $T_{1,(\sigma_{1,s},\tau_{1,s})}$ so that there are $\tau_{0,s+1}(k+|\tau_{0,s}|)$-many splitting nodes between $\sigma^{k-1}_{1,s+1}$ and $\sigma^k_{1,s+1}$. So  $\tau_{0,s+1}$ is coded.  Inductively, for $j\leq 1$, let $\sigma_{1,s+1}^{n^1_{s+1}+j+1}$ be the next splitting string in $T_{1,(\sigma_{1,s},\tau_{1,s})}$ extending $(\sigma_{1,s+1}^{n^1_{s+1}+j})^{\smallfrown}1$. This coding tells us that  the action at this stage for $ (\sigma_{1,s+1},\tau_{1,s+1})$-part is finished.   Define $\sigma_{1,s+1}=\sigma_{1,s+1}^{n^1_{s+1}+2}$. Let $\tau_{1,s+1}\in \omega^{|\sigma_{1,s+1}|}$ be the leftmost finite string so that the tree $T_{1,(\sigma_{1,s+1},\tau_{1,s+1})}$ has an infinite path. So we code $\tau_{0,s+1}$ into $\sigma_{1,s+1}$.

This finishes the construction at stage $s+1$.

\bigskip

Let $x_i=\bigcup_{s\in \omega}\sigma_{i,s}$  for $i\leq 1$. Obviously $z\geq_h x_0\oplus x_1$.

Now we use $x_0$ and $x_1$ to decode the coding construction. The decoding method is a finite injury which is quite similar to the new proof of Theorem \ref{theorem: martin and friedman theorem}. We want to construct a sequence ordinals $\{\alpha_s\}_{s\in \omega}$ $\Delta_1$-definable in $L_{\omega_1^{x_0\oplus x_1}}[x_0\oplus x_1]$ so that $\lim_{s\to \omega}\alpha_s=\CK$. Once this is done, then it is obvious to decode the construction and so $x_0\oplus x_1\geq_h z$.

As in the proof of Theorem \ref{theorem: martin and friedman theorem}, we may fix a $\Sigma_1$ enumeration of $\{T_i[\alpha]\}_{i\leq 2, \alpha<\CK}$ over $L_{\CK}$ so that for $i\leq 1$,
\begin{itemize}
\item $T_i[0]=T_i$; and
\item $T_{i}[\alpha]\subseteq T_i[\beta]$ for $\CK>\alpha\geq \beta$; and
\item $T_{i}[\CK]=\bigcap_{\alpha<\CK}T_i[\alpha]$; and
\item $T_{i}[\CK]$ has no dead node; and
\item $A_i=\{x\mid \exists f\forall n (x\uh n,f\uh n)\in T_i[\CK]\}$.
\end{itemize}
Since $[T_i]$ does not contain a hyperarithmetic infinite path, we have that $[T_i[\CK]]=[T_i]$ for $i\leq 2$.

We perform  almost the same decoding construction as in the proof of Theorem \ref{theorem: martin and friedman theorem}. At every stage $s$, we have a guess for the parameters defined in the coding construction up to stage $s$. Also we may need to correct our guess by searching a big ordinal $\alpha_s$ to redefine those parameters turning out to be incorrect at stage $\alpha_s$ (if necessary). Since neither $x_0$ nor $x_1$ is the leftmost real in $T_0$ or $T_1$, these parameters can only  be redefined at most finitely many times. So every parameter will be stable after stage $\alpha_s$ for large enough $s$. Then using the same method as in the proof of Theorem \ref{theorem: martin and friedman theorem}, we can find $f$ at stage $\alpha=\lim_{s\to \omega}\alpha_s$. The only extract effort is to decode $\tau_{i,s}$. But this is just like to decode $f(s)$ without any new insight.

So $\KO\equiv_h f\leq_h x_0\oplus x_1$.

\end{proof}
By \cite{cny07}, being low for $\QI$ randomness  is equivalent to being  low for $\DI$ randomness and being not cuppable above $\+ O$ by a $\QI$ random. So it suffices to define a low for $\DI$ random  that is cuppable.

\begin{lemma}\label{lemma: simpson low lemma}
There is an uncountable $\Sigma^1_1$ set $A$ in which every real is $\Delta^1_1$-traceable.
\end{lemma}
\begin{proof}
This follows directly  from the proof of Theorem 4.7 in \cite{Simpson:75}.
\end{proof}

By \cite{cny07}, each  $\Delta^1_1$-traceable real is  low for $\Delta^1_1$-random. By \cite{Hjorth.Nies:07}, the $\Pi^1_1$-random reals form a $\Sigma^1_1$ set. Then by Lemma \ref{lemma: simpson low lemma} and \ref{lemma: simga11 joint lemma}, there is a real $x$ which is low for $\Delta^1_1$-randomness  but $x\oplus y\geq_h \KO$ for some $\Pi^1_1$-random real $y$. We may conclude:
\begin{theorem}
Lowness for $\Delta^1_1$-randomness does not imply lowness for $\Pi^1_1$-randomness. And lowness for $\Delta^1_1$-Kurtz-randomness does not imply lowness for $\Pi^1_1$-Kurtz-randomness.\end{theorem}

{\bf Remark:} Lemma \ref{lemma: simga11 joint lemma} can be used to answer Question 58 in \cite{Friedman:75*1} and Question 3 in \cite{Simpson:75}. Solutions were announced  by Friedman and Harrington, but never published. 

\section{Some notes on $\Delta^1_2$- and $\Sigma^1_2$-randomness}
Input by Yu. (October)

\subsection{Within $ZFC$}
The ground model for $\Sigma^1_2$-theory is $L_{\delta^1_2}$ where $\delta^1_2$ is the least ordinal which can not be $\Delta^1_2$-definable. We have Gandy-Spector theorem for $\Sigma^1_2$-sets over $L_{\delta^1_2}$. The generalized Turing jumps with $L_{\delta^1_2}$ are $\Pi^1_1$-singletons.

Within $ZFC$, very limited interesting results can be obtained for the randomness notions. The following result was proved by Kechris under $PD$ but it turns out to be a theorem under $ZFC$.
\begin{theorem}[Kechris \cite{Kechris73}]\label{theorem: kechris sigma12}
Given a measurable $\Sigma^1_2$ set $A\subseteq 2^{\omega}$, both the sets $\{p\in \mathbb{Q}\mid \mu(A)>p\}$ and $\{p\in \mathbb{Q}\mid \mu(A)\geq p\}$ are $\Sigma^1_2$.
\end{theorem}

Again the following result was proved by Kechris under $PD$ which is unnecessary.
\begin{proposition}[Kechris \cite{Kechris73}]\label{proposition: appoxmiating simga12 sets}
Suppose that $A\subseteq 2^{\omega}$ is a $\Sigma^1_2$ set with a positive measure, then for any $n$, there is a $\Delta^1_2$  perfect set $B_n\subseteq A$ so that $\mu(A)<\mu(B_n)+2^{-n}$. Moreover, if $\mu(A)$ is $\Delta^1_2$, then $\{(n,x)\mid x\in B_n\wedge n\in \omega\}$ is $\Delta^1_2$.
\end{proposition}

An immediate conclusion of Proposition \ref{proposition: appoxmiating simga12 sets} is:
$$\Delta^1_2\mbox{-randomness}=\Delta^1_2\mbox{-ML-randomness}=\Pi^1_2\mbox{-randomness}.$$

By a similar enumeration over $L_{\delta^1_2}$, we have a universal $\Sigma^1_2$-ML-test. 
By a forcing argument over $\Delta^1_2$-closed positive measure sets, we may show that the collection of $\Delta^1_2$-random reals is not $\mathbf{\Sigma}^0_3$. So $\Delta^1_2$-randomness is different than $\Sigma^1_2$-ML-randomness.

Fix a $\Pi^1_2$-tree $T$ presenting a $\Pi^1_2$-closed set only containing $\Sigma^1_2$-ML-random reals.  The leftmost path of $T$ can be covered by a generalized $\Sigma^1_2$-ML-test. So $\Sigma^1_2$-ML-randomness is different with strong $\Sigma^1_2$-ML-randomness.

\subsection{Outside $ZFC$}

If $V=L$, then for every real $x$, the set $\{y\mid x\not\leq_{\Delta^1_2}y\}$ is countable and so null. But if we believe that $L$ must be small, then we have regular results as following.
\begin{theorem}[Kechris \cite{Kechris73}]\label{proposition: delta12 ordering is null}
 Suppose that every $\Sigma^1_2$ set of reals is measurable. If $x$ is a real so that $\{y\mid x\leq_{\Delta^1_2}y\}$ has positive measure, then $x$ is $\Delta^1_2$;
\end{theorem}
Theorem \ref{proposition: delta12 ordering is null} was proved by Kechris under $PD$. Stern \cite{Stern75} observed that it can be proved under the much weaker assumption: by Solovay \cite{Solovay70}, the assumption in the Theorem is a consequence of ``$(\omega_1)^L<\omega_1$".

\begin{definition}\index{Randomness!$L$-}
A real $x$ is {\em $L$-random} if for any Martin-L\" of test $\{U_n\}_{n\in \omega}$ in $L$, $x\not\in \bigcap_{n\in \omega}U_n$.
\end{definition}

$L$-randomness was essentially introduced by Solovay \cite{Solovay70}. The following fact is  obvious.
\begin{proposition}\label{proposition: basic fact about lrandom}
The followings are equivalent:
\begin{itemize}
 \item[(i)] $x$ is $L$-random;
 \item[(ii)] $x$ does not belong to any Borel null set coded in $L$;
 \item[(iii)] $x$ is a $\mathbb{R}=(\mathbf{T},\leq)$-generic real over $L$, where $\mathbb{R}$ is random forcing.
\end{itemize}
\end{proposition}

The following theorem is an analog of $\Pi^1_1$-randomness theory.
\begin{theorem}[Stern \cite{Stern75}]\label{theorem: stern theorem characterizing sigma12 random}
For any real $x$, $x$ is $\Delta^1_2$-random and $\delta^{1,x}_2=\delta^1_2$ if and only if $x$ is $L$-random;
\end{theorem}

If $L$ is small, then $\Sigma^1_2$-randomness is the same as $L$-randomness.
\begin{theorem}[Stern \cite{Stern75}]\label{theorem: largest sigma12 null set}
Assume that the set of  non-$L$-random reals is null. Then the set of non-$L$-random reals is the largest null $\Sigma^1_2$ set and so $\Sigma^1_2$-randomness is the same as $L$-randomness.
\end{theorem}

Solovay \cite{Solovay70} proves that the set of  non-$L$-random reals is null if and only if every $\Sigma^1_2$-set is measurable. So if every $\Sigma^1_2$-set is measurable, then strong $\Sigma^1_2$-ML randomness is different than $\Sigma^1_2$-randomness. This sheds some light on the corresponded $\Pi^1_1$-randomness problem.

\subsection{Lowness}
Very little  is known.

If $V=L$, then lowness for $\Delta^1_2$-randomness=lowness for $\Sigma^1_2$-ML-randomness=$\Delta^1_2$-ness. But they should not be treated as ``regular results".

Note that every Sacks generic real is low for $L$-random. So if every $\mathbf{\Sigma}^1_2$-set is measurable, then every Sacks generic real is low for $\Sigma^1_2$-random. Actually, they are precisely those reals that are ``constructibly traceable" under certain assumptions. 

\subsection{Within $L$}
Within $L$, nothing is interesting for higher-up randomness notions. The main point is that there is a $\Delta^1_2$-well ordering over reals in $L$. So for $n\geq 2$, to study $\Sigma^1_n$-randomness, we have to appeal to some axioms to make the universe   ``regular".

All the proofs  can be found in the forthcoming book \cite{Chong.Yu:nd}. Draft available on Yu's web site.

\newpage

\part{Complexity of Equivalence relations}

\section{Complexity of arithmetical equivalence relations}

\def \Intalg {\text{\it Intalg }}

Sy Friedman, Katia Fokina, Andr\'e Nies worked  Vienna, Jan.\ 2012.  This research also  involves  discussions at the  Oberwolfach February meeting with  D. Cenzer, J.\ Knight, J.\  Liu, V. Harizanov, A.  Nies. More recently (April 2012) Russell Miller and  Selwyn Ng  joined in this line of research, along with Nies' MSc student Egor Ianovski.

For equivalence relations $E,F$ with domain $\w$, we write $E \le_1 F$ if there is a computable 1-1 function  $g$ such that $xEy \LR g(x) F g(y)$.

\subsection{There is a $1$-complete $\PI 1$ equivalence relation}

\mbox{}

We begin with some examples of $\PI 1$ equivalence relations. \bi \item  Elementary equivalence of automatic structures for the same finite signature (according to Khoussainov, known to be undecidable by some Russian result); the same with the extended language allowing~$\ex^\infty$. 

 \item Isomorphism of automatic equivalence structures/ trees of height 2. This is $\PI 1$ complete in the set sense by Kuske, Liu and Lohrey    (TAMS, to appear), but not known to be $\PI 1$ complete for eqrels.

\item If $f$ is a binary computable function then let $E_fxy \lra \fa i \, f(x,i) = f(y,i)$, which clearly  is $\PI 1$. Ianovski, Miller, Nies and Ng~\cite{Ianovski.Miller.etal:nd} have shown that every $\PI 1$ eqrel is of  this form. This  contrasts with    Marchenkov's result  (1970s) that there is no universal   negative enumeration (Reference?). \ei

\begin{thm}  There is a $\PI 1$ equivalence relation $E$ such that 

\n $G \le_1 E$  for each  $\PI 1$ equivalence relation $G$.
\end{thm}

\begin{proof}
	 Let $X^{[2]}$ denote the unordered pairs  of elements of $X$. A set $R \sub X^{[2]}$ is \emph{transitive} if $\{u,v\}, \{v,w\} \in R$ and $u \neq w$ implies $\{u, w\} \in R$.  We view eqrels as transitive subsets of $\w^{[2]}$.   We view the $p$-th r.e.\ set $W_p$ as a subset of $\w^{[2]}$.
	
	 The idea is to copy $W_p$ as long as $\w^{[2]} - W_p$ looks transitive. The resulting partial copy  may have  finite or infinite domain. If $\w^{[2]} - W_p$ is indeed transitive then the domain is infinite.
	
	 In a sense $E$ is a uniform  disjoint sum of all these partial copies. 
	
	Uniformly in a given $p$ define a partial computable sequence of stages by $t^p_0 = 0$ and   
	
	\bc $t^{p}_{i+1} \simeq \mu t > t^p_i \, \big [  [0, t^p_i)^{[2]} - W_{p,t} \text{ is transitive }\big ]$. \ec
	
	Define a computable function  $g \colon \, \w \to \w \times \w$ by the following construction. At each stage the domain of $g$ is a finite initial segment of $\w$. At stage $t$, for each $i$ with $0< i \le t$, if $t = t^p_{i+1}$ then add an interval of numbers $L^p_i$ of length $\ell = t^p_i- t^p_{i-1}$ to the domain of $g$, and map it in an increasing fashion to the least  $\ell$ numbers in $\{ p \} \times \w$ that are not yet in the range of $g$. In this way, we extend the range of $g$ to contain $\{ p \} \times [0,t^p_i)$.  
	
	Now for $x<y$ declare  $\{x,y\} \in E$ if $g(x)_0 = g(y)_0= :p$, and for the unique $i$ such that $y\in L^p_i$, we have  \[ \fa t \ge t^p_i \, \fa k \ge i  \,  [ t= t^p_{k+1} \to \{g(x)_1, g(y)_1 \}\not \in W_{p,t}]. \] 
	Clearly $E$ is $\PI 1$. To check $E$ is transitive, suppose $\{x,y\} \in E$, $\{z,y\} \in E$, where $x<y$, and $z< y$.  We may  suppose that $x< z$. There are unique $p,a,b,c$ such that $g(x) = \la p,a\ra, g(y) = \la p,b \ra$, and $g(z) = \la p,c \ra $. We have  $y \in L^p_i$, and $z \in L^p_r$ for  some $r \le i$. Assume that there is $k \ge r$ such that for $t = t^p_{k+1} $ we have $ \{a, c \} \in W_{p,t}$. Let  $j = \max(i,k)$ and note that $t^p_{j+1}$ is defined.  Then  $[0, t^p_j)^{[2]} - W_{p,t^p_{j+1}}$ is not transitive, because it contains $\{a,b\}, \{b, c \}$, but not $\{a,c \}$. This contradicts the definition of $t^p_{j+1}$. Thus $\{x,z\} \in E$.

	Finally, given a  $\PI 1$ equivalence relation $G= \w^{[2]}-W_p$, the sequence of stages $(t^p_i)\sN i$ is infinite. Hence, the (total) computable map $a \to g^{-1}(\la p, a\ra)$ is the required $1$-reduction of $G$ to $E$.
\end{proof}

\subsection{Completeness  for $\SI 3$  equivalence relations  of  $\equiv_1$ on the  r.e.\ sets }

 It is trivial that for each there is  some $\SI n$ complete equivalence relation, because the  transitive closure of a $\SI n$ relation  is $\SI n$. We can make it unique by requiring it to be EUH (effective universal homogeneous) in the sense of Nerode and Remmel. 
 
 In the following we look for natural examples. For $\SI 3$ we don't have to look far.
\begin{thm} \label{thm:1-equiv-complete} For each $\SI 3$ equivalence relation $S$, there is a computable function $g$ such that 

\begin{eqnarray*}
   yS z   &   \RA &  W_{g(y) } \equiv_1 W_{g(z)}, \, \text{and} \\
     \lnot yS z   & \RA & W_{g(y) }, W_{g(z)} \, \text{are Turing incomparable.}  
\end{eqnarray*}
 \end{thm}

\begin{corollary} \label{cor:m_Sigma_complete}  Many-one equivalence and 1-equivalence on indices of c.e.\ sets are $\SI 3$ complete for equivalence relations  under  computable reducibility. \end{corollary}
According to S. Podzorov at the Sobolev Institute Novoskibirsk, the result for m-equivalence  was possibly known to the Russians by the end of the 1970s. No reference has been given   yet.  Sadly, Podzorov passed away in late 2012.

Note that this is significantly stronger than the mere  $\SI 3$ completeness of $\equiv_m$ as a set of pairs of c.e.\ indices, which follows for instance because the $m$-complete c.e.\ set have a $\SI 3$ complete index set.  
As a further consequence, Turing equivalence on indices of c.e.\ sets is a $\SI 3 $ hard equivalence relation for computable reducibility. However, this equivalence relation is only $\SI 4$. Ianovski et al.\ \cite{Ianovski.Miller.etal:nd}  have shown that in fact it is    $\SI 4$ complete in our sense.

\begin{proof}[Proof of theorem]

 \label{s:proof_main}

Since $S$ is $\SI 3$,  there is a uniformly c.e.\ triple sequence \bc  $(V_{y,z,i})_{y,z,i \in \w, y<z}$  \ec of initial segments of $\NN$ such that for each $y<z$, \bc $ySz \LR  \ex i \, V_{y,z,i} = \w$. \ec 

We build a uniformly c.e.\ sequence of sets $A_x = W_{g(x)}$ ($x \in \w$), $g$ computable. We meet  the following  coding requirements for all $y<z $ and $i \in \w$.

\vsps

 $G_{y,z,i} \colon \, V_{y,z,i} = \w \RA  A_y \equiv_1 A_z$. 

\vsps
 
\n  We meet    diagonalization requirements for $u\neq  v$,

\vsps


 $N_{u,v,e} \colon  u = \min [u]_S \lland v = \min [v]_S \RA$ 
   $ A_u \neq \Phi_e(A_v)$.

\vsps
\n where $\Phi_e$ is the $e$-th Turing functional,  and $[x]_S$ denotes the $S$-equivalence class of $x$.
Meeting these requirements suffices to establish the theorem.
 
The basic strategies to meet the requirements are as follows. If $V_{y,z,i} = \w$, a strategy for $G_{y,z,i} $     ``finds out'' that $z$ is $S$-related to the smaller $y$. Hence it  builds a computable permutation~$h$ such that  $  A_y \equiv_1 A_z$ via~$h$.  

  A strategy for  $N_{u,v,e} $ picks a witness~$n$, and waits for $\Phi_e(A_v; n)$ to converge. Thereafter, it  ensures that   this computation is stable and  $A_u(n) $ does not equal its  output $\Phi_e(A_v; n)$ by enumerating $n$ into $A_u$ if this output is~$0$.

\vsp

\n \emph{The tree of strategies.} To avoid conflicts between strategies that enumerate into the same set $A_z$, we need to provide the strategies with a guess at whether $z$ is least  in its $S$-equivalence class $[z]_S$. An $N$-type strategy will only enumerates into $A_z$ if according to its guess,  $z$ is least in its $[z]_S$; a $G$-type strategy only enumerates into $A_z$ if according to its guess, $z$ is not least. 

 Fix an effective priority ordering of all requirements. We define a  tree $T$ of strategies, which is a computable subtree $T$ of $\strcantor$.  We write $\aaa: R$ if strategy $\aaa$ is associated with the requirement $R$.  By recursion on $|\aaa|$,  we define whether $\aaa \in T$,  and which is  the requirement  associated with $\aaa$.  We also define a function $L$ mapping $\aaa \in T$ to a cofinite set $L(\aaa)$ consisting of the numbers~$x$ such that according to  $\aaa$'s guesses, $x$ is least in its equivalence class. 

Let $L(\estring) = \w$. Assign to $\aaa$ the highest priority requirement $R$  not yet assigned to a proper prefix of $\aaa$ such that either  (a) or (b) hold.

\bi 
\item[(a)] $R$ is $G_{y,z,i}$ and $z  \in L(\aaa)$; in this case put both $\aaa 0 $ and $\aaa 1$ on $T$, and define $L(\aaa 0) = L(\aaa)- \{z\}$ while  $L(\aaa 1) = L(\aaa) $ (along $\aaa 0$ we know that $x$ is no longer the least in its equivalence class)

  \item[(b)]  $R$ is $N_{u,v,e}$ and $u,v \in L(\aaa)$; in this case put  only  $\aaa 0 $   on $T$, and define $L(\aaa 0) = L(\aaa) $. 
\ei

For strings $\aaa, \bbb \in \strcantor$, we write $\aaa <_L \bbb$ if there is $i$ such that $\aaa\uhr i = \bbb\uhr i $, $\aaa(i)=0$ and $\bbb(i)=1$. We let  $\aaa \preceq \bbb$ denote that $\aaa$ is a prefix of $\bbb$. We define a linear ordering on strings  by \bc  $\aaa \le  \bbb$ if $\aaa<_L \bbb$ or $\aaa \preceq \bbb$. \ec

\n \emph{Construction of a u.c.e.\  sequence of sets   $(A_x)\sN x$.} 
We   declare in advance that  $A_x(4m+1)= 0$ and $A_x(4m+3)= 1$ for each $x,m$.	 The construction then only determines membership of even numbers in the $A_x$.

We define a computable sequence $(\delta_s)\sN s$  of strings on $T$ of length $s$. Suppose inductively that $\delta_t$ has been defined for $t< s$. Suppose $k<s$ and that $\eta = \delta_s \uhr k$ has been defined. If $\eta: N_{u,v,e}$ let $\delta_s(k) = 0$. Otherwise $\eta: G_{y,z,i}$. Let $t< s$ be the largest stage such that $t=0$ or $\eta \preceq \delta_t$. Let $\delta_s(k)=0 $ if $V_{y,z,i,s}\neq V_{y,z,i,t}$, and otherwise $\delta_s(k)=1$. 

The \emph{true path} $TP$ is the lexicographically leftmost path $f \in \tp \w$ such that $\forall n \, \exists^\infty s \ge n \, [ \delta_s\uhr n \prec f]$.
 To \emph{initialize} a strategy $\aaa$ means to return it to its first instruction. If $\aaa: G_{y,z,i}$ we also make the partial computable function  $h_\aaa$ built by  the strategy $\aaa$   undefined on all inputs. At stage $s$, let $\init(\aaa,s)$ denote the largest stage $\le s$ at which $\aaa$ was initialized.

\vsps

\n \emph{An   $N_{u,v,e}$ strategy  $\aaa$.} 
At stages $s$:
\bi 
\item[(a)]  {Appoint} an unused  even number $n> \init(\aaa,s)$ as a  witness for diagonalization.     {Initialize} all the  strategies $\beta \succ \aaa$.  
\item[(b)] 
 {Wait} for  $\Phi_e(A_v;n) [s]$ to converge with output $r$.   If $r=0$ then  {put} $n$ into $A_u$.   {Initialize} all the strategies $\beta \succ \aaa$.   
\ei

  \vsp
  
  \n \emph{A  $G_{y,z,i}$ strategy  $\aaa$.}
    If $\aaa 0$ is on the true path then this strategy builds a    computable increasing map $h_\aaa$ from even numbers to even numbers such that  $A_y(k) = A_z(h_\aaa(k))$ for each~$k$. Furthermore, $A_z - \range  (h_\aaa)$ is computable. By our definitions  of $A_y$ and $A_z$ on the odd numbers, this implies that  $h_\aaa$ can be extended  to a computable permutation showing that $A_y \equiv_1 A_z$,  as required.

At stages $s$, if $\aaa0 \sub \delta_s$, let $t<s$ be greatest such that $t=0$ or $\aaa0 \sub \delta_t$, and do the following.
\bi 
\item[(a)]   For each even $k<s$ such that  $k \not \in   \dom (h_{\aaa, t})$ pick an unused even value $m= h_{\aaa, s}(k)>  \init(\aaa,s) $ in such a way that $h_\aaa$ remains increasing. 

\item[(b)]  From now on, unless $\aaa$ is initialized,   ensure that $A_z(m) =A_y(k)$.  (We will verify that this is possible.)

\ei

\n The stage-by-stage construction is as follows.   
At stage $s>0$ initialize all strategies $\aaa >_L \delta_s$.    Go through substages $i \le s$. Let $\aaa = \delta_s\uhr i$. Carry out  the  strategy~$\aaa$ at stage~$s$.

\vsp

\verif  
To show the requirements are met,  we first   check that  there is no conflict between different  strategies that enumerate into the same set $A_z$.
\begin{claim} Let $\aaa \colon G_{y,z,i}$. Then (b) in the strategy for $\aaa$ can be maintained as long as $\aaa$ is not initialized.   \end{claim}

To prove the claim, suppose a strategy $\bbb\neq \aaa$ also enumerates numbers into~$A_z$. If $\aaa0 <_L \bbb$ then $\bbb$ is initialized when $\aaa$ extends its map $h_\aaa$, so the numbers enumerated by $\bbb$ are not in the range of $h_\aaa$. If $\bbb<_L \aaa0$ then $\aaa$ is initialized when $\bbb$ is active, so again   the numbers enumerated by $\bbb$ are not in the range of $h_\aaa$.  Now suppose  neither hypothesis holds, so $\aaa0 \preceq \bbb$ or $\beta \prec \aaa$. 

\vsps

\n \emph{Case $\bbb \colon N_{z,v,e}$.}   In this case  $\aaa0 \preceq \bbb$ is not possible  because $z \not \in L(\aaa 0)$. If $\bbb \prec \aaa$ then $\aaa$ is initialized when $\beta$ appoints a new diagonalization witness.

\n \emph{Case $\bbb \colon G_{y', z, i'}$.} In this case
 $\aaa0 \preceq \bbb$ is not possible because $z \not \in L(\aaa 0)$.  
If $\beta 1 \preceq \aaa$ then $\aaa$ is initialized each time $\beta$ extends its map $h_\beta$. Finally, 
$\beta0 \preceq \aaa$ is not possible because $z \not \in L(\beta 0)$. This proves the claim.

\begin{claim}  Let $\aaa$ be the $N_{u,v,e}$   strategy on the true path.   Suppose  $\aaa$ is not initialized after stage $s$. Then  $\aaa$  only acts finitely often, and  meets its requirement. \end{claim}
   At some stage $\ge \init(\aaa,s)$ the strategy $\aaa$  picks a permanent witness  $n$. No    strategy $\beta \prec \aaa$ can put $n$ into $A_u$ because $u \in L(\aaa)$. No other  strategy can put  $n$ into $A_u$ because of the initialization $\aaa$ carries out when it picks $n$.  Suppose now  that at a later stage~$t$, a computation $ \Phi_e(A_v;n) [t] $ converges. Since $v \in L(\aaa)$, no $G$-type strategy $\beta \prec \aaa$ enumerates into $A_v$. Thus the initialization of strategies $\gamma \succ \aaa$ carried out by~$\aaa$ at that stage~$t$ will ensure that this computation is preserved with value different from $A_u(n)$. This proves the claim.

It is now clear by induction that each strategy $\aaa$  on the true path is initialized only finitely often. Thus the $N$-type requirements are met. Now suppose $\aaa \colon G_{y,z,i}$ and   $\aaa0 $ is on the true path. Then no strategy $\beta\succeq \aaa0$ enumerates into $A_z$. Thus by the initialization at stages $s$ such that $\aaa0 \preceq \delta_s$, the set $A_z - \range  (h_\aaa)$ is computable. As noted earlier, this implies that  $h_\aaa$ can be extended to a computable permutation showing that $A_y \equiv_1 A_z$. There is a  computable bijection $q$  between the set   of odd numbers and the set of numbers that are odd,  or even but not in the range of $h_\alpha$, so that  $m \in A_y \lra q(m) \in   A_z$. Now let the permutation be $q \cup h_\aaa$.

\end{proof}


\subsection{Computable isomorphism of trees}
We use the terminology  of Fokina  et al. \cite{Fokina.Friedman.etal:12}.  Thus,  a tree is a structure  in the language containing  the predecessor function as a single unary function symbol.  The root is its own predecessor.   A countable tree can be represented given by a nonempty  subset  $B$   of $\omega^{< \omega}$ closed under prefixes, where the predecessor function takes of the last entry of a non-empty tuple of natural numbers.

A tree  has a  computable presentation  iff  we can choose  $B$  r.e.\ For in that case $B $ is the range of a partial computable 1-1 function $\phi$ with domain an initial segment of $\omega$; the preimage of the predecessor function under  $\phi$ is the required computable atomic diagram. 

We let $B_e = \{ \sss\colon \, \ex \tau \succeq \sss \,[ \tau \in W_e]  \}$, where the $e$-th r.e.\ set $W_e$ is viewed as a subset of $\omega^{< \omega}$. Then $(B_e)\sN e$ is a uniform listing of all computable trees.  

We say a tree has height $k$ if every leaf has length at most $k$.  


\begin{cor} \label{cor:tree_comp_isom} Computable isomorphism of computable  trees of height $2$ where every node at level $1$ has out-degree at most $1$ is  $m$-complete for $\SI 3$ equivalence relations.  \end{cor}

\begin{proof} Let $h$ be a computable function such  for each $e$,    $B_{h(e)}$ is the tree 
	\bc $\estring \cup \{\la x \ra \colon \, x\in \w \} \cup \{\la x, 0 \ra \colon \, x\in W_e\}$. \ec Clearly, $W_y \equiv_1 W_z$ iff $B_{h(y)}$ is computably isomorphic to $B_{h(z)}$. Now we apply Theorem~\ref{thm:1-equiv-complete}.   
\end{proof}

For background on computable metric spaces, see \cite{Brattka.Hertling.ea:08}.   A computable metric space is \emph{discrete} if every point is isolated. For such a space, necessarily every point is an ideal  point.

\begin{cor} Computable isometry  of discrete computable metric spaces is $m$-complete for $\SI 3$ equivalence relations. \end{cor}
\begin{proof}  
	Given a  computable  tree $B$,  create a discrete computable metric $M_B$ space as follows: if a string $\la x \ra$ enters $B$, add a point $p_x$. If later $\la x,i \ra$ enters $B$ for the first  $i$, add a further point $q_x$. Declare $d(p_x,q_x)= 1/4$. Declare $d(p_x, p_y)=1$ and $d(q_x,p_y) =1 $ (if $q_x$ exists). Clearly for trees $B,C$ as in Cor.\ \ref{cor:tree_comp_isom}, $B$ is computably isomorphic to $C$ iff $M_B$ is computably isometric to $M_C$.   \end{proof}

\begin{cor} Computable isomorphism of recursive equivalence relations where every class has at most 2 members is $m$-complete for $\SI 3$ equivalence relations. \end{cor}
\begin{proof}  Given r.e.\ set $A$, build a  {computable} equivalence relation $R_A$ such that \bc  $A \equiv_1 B$ iff $R_A \equiv_{comp} R_B$.  \ec

We may assume at most one element enters $A$ at each stage, and only at even stages.

Let $R_A= \{ \la 2a, 2t+1 \ra \colon a \ \text{enters $A$ at stage} \ t \}$.
\end{proof}

\subsection{Boolean algebras}

For a linear order $L$ with least   element, $\Intalg L$ denotes  the   subalgebra of the Boolean algebra $\+ P(L)$ generated by intervals $[a,b)$ of $L$ where $a \in L $ and $b \in L \cup \{\infty \}$. Here $\infty$ is a new element greater than any element of $L$, and $[a, \infty)$ is short for $\{ x \in L \colon \, x \ge a\}$. Note that $\Intalg L$ consists of all sets $S$  of the form 
\[ S = \bigcup_{r=1}^n [a_r, b_r)\]
where $a_0 < b_0 < a_1 \ldots < b_n \le \infty$.
From a computable presentation  of  $L$ as a   as a  linear order,    we may  canonically obtain a computable presentation of the Boolean algebra  $\Intalg L$.

\begin{theorem} Computable isomorphism of computable Boolean algebras is  complete for $\SI 3$ equivalence relations.\end{theorem}
	
	\begin{proof}  Let $(V^e)\sN e$ be an effective listing of the c.e.\ sets containing the even numbers. 
 The relation of  $1$-equivalence $\equiv_1$  of c.e.\ sets $V^e$  is $\SI 3 $ complete by Theorem~\ref{thm:1-equiv-complete} and its proof below. We 	will computably reduce it to computable isomorphism of computable Boolean algebras.
	We define the Boolean algebra $C^e $ to be the interval algebra of a computable linear order $L^e$. Informally, to define $L^e$, we begin with  the order type $\omega$. For each $x\in \omega$, when $x$ enters $V^k$ we replace $x$  by a computable copy of $[0,1)_\QQ$. More formally,  \bc $L^e = \bigoplus_{x \in \omega} M^e_x$,  \ec 
	where
	$M^e_x$ has one element $m^k_x = 2x $, until $x$ enters $V^e$; 
	if and when that happens, we expand $M^e_x$ to  a computable  copy of $[0,1)_\QQ$, using the odd numbers,    	 
	  while ensuring that $m^k_x = \min M^k_x$  holds in $L^k$.  Also note that the domain of $L^k$ is $\NN$ because $0 \in V^k$. 
	 
	\begin{claim} $V^e \equiv_1 V^i$ $\LR $ $C^e \cong_{comp} C^i$. \end{claim}

		\rapf  
	Suppose $V^e \equiv_1 V^i$ via a computable permutation $\pi$. We define a computable isomorphism $\Phi: C^e \cong C^i$.
	
	\n (a)  Let $\Phi(m^e_x) = m^i_{\pi(x)}$. Once $x$  enters $V^e$, we know that $\pi(x) \in V^i$. So we may always ensure that  $\Phi$ restricts to  a computable isomorphism of linear orders $M^e_x \cong M^i_{\pi(x)}$.

\n (b) Consider an  element $S$  of $C^e$. It is given in the form $S = \bigcup_{r=1}^n [a_r, b_r) $ where $a_0 < b_0 < a_1 \ldots < b_n $ for $a_r, b_r \in L^e \cup \{\infty\} $ as above.  If $b_n< \infty$, we can compute the  maximal $x\in \omega$  such that $M^e_x \cap S \neq \ES$. Define 
\[ \Phi(S) = \bigcup_{y \le x}\Phi(S \cap M^e_y).\]
Note that the set $\Phi(S \cap M^e_y)$ can  be determined by (a).

If $b_n = \infty$, then let $\Phi(S) $ be the complement in $L^i$ of $\Phi (L^e \setminus S)$.

\vsp

\lapf Now suppose that $  C^e \cong_{comp} C^i$ via  some computable isomorphism $\Phi$. We show that $V^e \le_1 V^i$ via some computable function~$f$. Suppose we have defined $f(y) $ for $y< x$. We have  $\Phi(M^e_x) = \bigcup_{r=1}^n [a_r, b_r) $ where    $a_r, b_r \in L^i \cup \{\infty\} $ as above.

If $n>1$ then $M^e_x$ is not an atom in $C^e$, whence $x \in V^e$. Thus let $f(x) $ be the least even number that does not equal  $f(y)$ for any  $y< x$. 

Now suppose $n=1$. If $a_1 = m^i_y, b_1 = m^i_{y+1}$ then let $f(x) = y$. Otherwise, again we know $M^e_x$ is not an atom in $C^e$, and define $f(x)$ as before.

By symmetry, we also have $V^i \le_1 V^e$, and hence $V^i \equiv_1 V^e$ by Myhill's theorem. 	\end{proof}

Now the reader might be ready to conclude that for every reasonably rich  class of structures the computable isomorphism problem is $\SI 3$ complete for eqrels. But this is not so. For instance, consider the class of computable permutations of order 2 (this class itself is $\PI 2$). Then the  computable isomorphism  relation   on this class is $\PI 2$. (And this is  the same as the classical isomorphism relation.) This is so because we only need to figure out whether for two given permutations, both have the same  number of  1 cycles, and the same number  of 2 cycles. 

\subsection{Almost inclusion of r.e.\ sets  is a $\SI 3$-complete preordering}
For $X, Y \sub \omega$, we write $X \sub^*Y$ if $X \setminus Y$ is finite. We write $X = ^* Y $ if $X \sub^* Y \sub^* X$. 
Let $W_e$ denote the $e$-th r.e.\ set. 

\begin{theorem}\label{thm:starequal} $\{\la e, i \ra \colon W_e \sub^* W_i\}$ is $m$-complete for  $\SI 3$ preorderings.
\end{theorem}

\begin{proof}   All sets in this proof will  be  r.e. Fix  a non-recursive set $A$. By $X \sqsubseteq A$ we denote  that $X$ is a split of $A$, i.e., $A \setminus X$ is r.e. Let $X,Y$ range over splits of $A$.

Since $A$ is non-recursive, there is a small major subset $D \subset_{sm} A$ (see \cite[pg.\ 194]{Soare:87}). Then, for each  $X \sqsubseteq A$, we have
 \bc $X \sub^* D \lra X$ is recursive  \ec
 (see \cite[Lemma 4.1.2]{Nies:habil}). 
Consider the     Boolean algebra 
\bc $\+ B_D(A)=\{(X \cup D)^* \colon X \sqsubseteq A\}$, \ec
which has a  canonical $\SI 3$ presentation  in the sense of \cite[Section 2]{Nies:97*1}.
   The Friedberg splitting theorem implies that every nonrecursive set can be split into two nonrecursive sets obtained uniformlhy in an r.e.\ index for the given set. Iterating this, we obtain a   uniformly r.e.\ sequence of splittings $X_n \sqsubset A$ (given by r.e.\  indices for both the set and its complement in $A$) such that  the  sequence $(p_n)\sN n$ freely generates a subalgebra $\+ F$ of $\+ B_D(A)$, where  $p_n= (X_n \cup D)^*$. 

 Now suppose that $\preceq$ is an arbitrary  $\SI 3$ preordering. Let $\+ I_0 $ be the  ideal  of $\+ F$ generated by $\{ p_n - p_k \colon n \preceq k \}$. We claim that 
 
 \bc $ n \preceq k \lra  p_n - p_k  \in \+ I_0 $.\ec
The implication ``$\rightarrow$'' is clear by definition. For the implication ``$\leftarrow$, let $\+ B_\preceq$ be the Boolean algebra generated by the subsets of $\omega$  of the form  $\hat i = \{ r \colon \, r \preceq i\}$. The map $p_i \mapsto \hat i$ extends to a Boolean algebra homomorphism $g \colon \, \+ F \to \+ B_\preceq$ that sends $\+ I_0 $ to $0$. If $n \not \preceq k$ then $\hat n \not \subseteq \hat k$, and hence $p_n - p_k \not \in \+ I_0$. This proves the claim.

Now let $\+ I$ be the  ideal  of $\+ B_D(A)$ generated by $\+ I_0$. Clearly $\+ I_0 = \+ I \cap \+ F$. Since $p_n= (X_n \cup D)^*$, the claim now implies that 
 \bc $ n \preceq k \lra  ((X_n - X_k)\cup D)^*  \in \+ I $.\ec

Note that $\+ I$ is a $\SI 3$ ideal.  By the basic  $\SI 3 $ case of  the  ideal definability lemma in \cite{Harrington.Nies:98} (a~simpler proof of this case was given in \cite[Lemma 3.2]{Nies:97*1})  there is $B \in [D, A]$ such that  \bc $(Y\cup D)^* \in \+ I \lra  Y \sub^* B$ \ec
for each $Y \sqsubseteq A$.  
  Thus
   \bc $n \preceq k \lra ((X_n \setminus X_k)\cup D)^* \in \+ I  \lra X_n \sub^* X_k \cup B$. \ec
   Since the sequence of splittings  $(X_n)\sN n$ is uniform, this yields the desired $m$-reduction.
  \end{proof}

Each equivalence relation is   a preorder. Thus, as  an immediate consequence, we obtain:

\begin{cor}\label{cor:=*  Sigma 3}  $\{\la e, i \ra \colon W_e =^* W_i\}$ is $m$-complete for  $\SI 3$ equivalence relations.
\end{cor}	
   
For a natural complete $\SI 3$ preordering, one can consider  {embeddability of subgroups of $(\QQ,+)$}.

\begin{cor}\label{cor:EmbedAbGroups} Computable embeddability among computable subgroups of $(\QQ, +)$ is $m$-complete for $\SI 3$ preorderings.
\end{cor}

\begin{proof}
We will represent a computable group by a 4-tuple of computable functions, $(e,\oplus,\ominus,I)$, where $e(x)=1$ for the identity element and 0 elsewhere,
$\oplus$ is the binary group operation, $\ominus$ a unary function taking an element to its inverse and $I$ is the ``interpretation" function which maps natural numbers to $\QQ$.
Thereby, $\oplus(x,y)=z$ iff $I(x)+I(y)=I(z)$.
We do not require that $I$ be one to one.

Let $p_n$ be the $n$-th prime. For convenience, treat $p_0$ as 1. Code the $x$-th r.e.\ set $W_x$ into $G_x$: the group generated by  $1$  and  all $1/p_n$, $n \in W_x$.
For a computable presentation, let $x_r$ denote the last element to enter $W_{x,r}$. Use the odd numbers to encode all finite sequences of integers,
non-zero even numbers of the form $2r$ to encode $1/p_{x_r}$ and 0 to encode 0.
This immediately defines the behaviour of $I$ on even numbers, so to account for the odd map the
sequence $\la m_1,m_2,\dots,m_k\ra$ to $m_1/p_{x_1}+m_2/p_{x_2}\dots+m_k/p_{x_k}$.
Because $I$ so defined is totally computable, it implies computable $e,\oplus,\ominus$ via $I^{-1}$.
For instance, to compute $\oplus(x,y)$ first compute $I(x)$ and $I(y)$, and then find the least element of $I^{-1}(I(x)+I(y))$. In this vein for the rest of this proof we will abuse notation slightly and
interpret $I^{-1}(x)$ as the least element of $I^{-1}(x)$.

Now suppose $W_x$ is almost contained in $W_y$. We wish to show that $G_x$ is embeddable in $G_y$.
Let $P$ be the product of the finitely many primes $p_n$ with $n \in W_x \setminus W_y$, $I$ the interpretation function of $G_x$ and $J$ the interpretation function of $G_y$.

The desired embedding of $G_x$ in $G_y$ is given by $a\mapsto J^{-1}(PI(b))$.
Observe that since the group operations modulo $I$ or $J$ respectively correspond to addition of the rational numbers, it follows that
this mapping is one to one as $PI(a)=PI(b)$ if and only if $I(a)=I(b)$, and it preserves the group operation as $PI(a)+PI(b)=PI(\oplus(a,b))$.

On the other hand, suppose there are infinitely many elements in $W_x$ that are not in $W_y$.
Note that we can define a notion of divisibility in a group in the usual way: $a|b$ iff $I(b)=cI(a)$.
Any embedding of $G_x$ into $G_y$ must clearly preserve divisibility modulo the interpretation. That is, where $f$ is the embedding if $I(b)=cI(a)$ then $J(f(b))=cJ(f(a))$.
Observe that since $1=p_i/p_i$ for any $i$, for every $i\in W_x$, $2i|I^{-1}(1)$. We will show that this implies that no embedding is possible.

Let $i\in W_x\setminus W_y$. Let $f(I^{-1}(1))=J^{-1}(a/b)$. As $a/b=p_iI(f(1/p_i))$, and $p_i$ cannot appear in the denominator of any element in $G_y$, $p_i$ must appear in $a$.
However, as there are infinitely many such $i$, all of them coprime, no finite nominator can satisfy this requirement.\end{proof}

\subsection{Computable isomorphism versus classical non-isomorphism}

Note that for any two  infinite and co-infinite r.e.\ sets $W_y$ and $W_z$, the tree  $B_{h(y)}$ in Cor.\ \ref{cor:tree_comp_isom} is classically, but not computably,  isomorphic to $B_{h(z)}$. We now strengthen the result for trees: non-equivalence even turns into  classical non-isomorphism.
\begin{conjtheorem} For each $\SI 3$ equivalence relation $S$, there is a computable function $g$ such that  \bc $xS y \RA B_{g(x)} \cong_c B_{g(y)}$,   \ec 
and    \bc $ \lnot  xS y \RA B_{g(x)} \not \cong B_{g(y)}$ \ec\end{conjtheorem}

\begin{proof} (Sketch) 
As before we meet   coding requirements for $y<z$, 

\vsps

 $G_{y,z,i} \colon \, V_{y,z,i} = \w \RA  T_y \cong_c T_z$. 

\vsps

If $\lnot u S v$ we  ensure that $T_u $ is not  isomorphic to $T_v$. To do so, if  $u = \min [u]_S$ then $T_u$ will have  exactly $u$ infinite paths. They are   denoted \bc $f^u_0, \ldots, f^u_{u-1} $. \ec We meet the requirements

 $N_{u,e} \colon u = \min [u]_S \RA   f^u_0 \uhr {e+1}, \ldots, f^u_{u-1}\uhr {e+1}$ are defined.

As before there are     strategies $\aaa \colon N_{u,e}$.  At each stage $s$ we have approximations $f^u_{i,\aaa,s}$ of $f^u_i$ of length $e+1$, where $f^u_{i,\aaa,s} \prec f^u_{i,\bbb,s}$ for strategies $\aaa \prec \bbb$. 
 We let $f^u_{i,\ES,s} = \la 2i \ra$ for each~$s$, which makes the paths   distinct.

Fix an effective priority ordering of all requirements with $N_{u,e}< N_{u,e+1}$.  We define a tree $T$ of strategies and $L$ as before, with the only difference that the former  requirements  $N_{u,v,e}$ now 
become $N_{u,e}$.

 To \emph{initialize} a strategy $\aaa$ means to return it to its first instruction. If $\aaa: G_{y,z,i}$ we also make the partial computable function  $h_\aaa$ built by  the strategy $\aaa$   undefined on all inputs. If $\aaa: N_{u,e}$ we make all current  $f^u_{i,\aaa}$ undefined. As before,  at stage $s$, let $\init(\aaa,s)$ denote the largest stage $\le s$ at which $\aaa$ was initialized.

\n  \emph{Strategy} $\aaa\colon N_{u,e}$   at stage $s$.

 If $e>0$ let $\gamma \prec \aaa$ be the $N_{u,e-1}$ strategy, otherwise $\gamma = \ES$. If $f^u_{i,\aaa,s-1}$ is undefined, 
pick a large  even number $n$ so that for each $i<u$, $f^u_{i,\aaa,s}: = f^u_{i,\gamma ,s}\ape n > \init(\aaa,s)$. 
   {Initialize} all the $G$-type  strategies $\beta \succ \aaa$.

  \vsp

  \n \emph{A  $G_{y,z,i}$ strategy  $\aaa$ at stages $s$.}
    If $\aaa 0$ is on the true path then this strategy builds a    computable  1-1 map $h_\aaa$ from all   even strings in $T_y$    to even strings  in $T_z$ preserving the  length and the  prefix relation.

If $\aaa0 \sub \delta_s$, let $t<s$ be greatest such that $t=0$ or $\aaa0 \sub \delta_t$, and do the following.
 
Declare each $\eta \in T_{z,s-1}$, $\init(\aaa,s) < \eta $ unextendable (leaf). 
  For each even string  $\eta <s$ let $k$ be largest such that  $\eta\uhr k    \in   \dom (h_{\aaa, t})$. If $k< |\eta|$, pick a fresh even extension $\sss \succeq h_{\aaa,t}(\eta \uhr k)$, where $\sss   >  \init(\aaa,s) $ and  $|\sss| = |\eta|$, and let $h_\aaa(\eta \uhr j ) = \sss \uhr j$ for each $j$ with $k< j \le |\eta|$.
 
	The stage-by-stage construction follows the same scheme as before. In particular,    
	at stage $s>0$ initialize all strategies $\aaa >_L \delta_s$.
\end{proof}

 \newpage
 
 \part{Others}

\section{Bernstein v.s Vitali}
Input by Yu.

This is a result for fun concerning the question in: 

\href{http://mathoverflow.net/questions/71575/vitali-sets-vs-bernstein-sets}{http://mathoverflow.net/questions/71575/vitali-sets-vs-bernstein-sets}

I need to modify the definition of Vitali set to apply recursion theory.  

We call a set $V\subset 2^{\omega}$ to be {\em Vitali} if for any Turing degree $\mathbf{x}$, there is a unique real $x\in \mathbf{x}\cap V$.

We call a set    $V\subset 2^{\omega}$ to be {\em Bernstein} if neither $V$ nor $2^{\omega}\setminus V$  contains a perfect subset.

Both Vitali and Bernstein sets are used   construct nonmeasurable sets as in classical analysis books.

Now an interesting question is which way is stronger? More precisely, over $ZF$, does the existence either one implies the existence of another one? 

\begin{theorem}\label{theorem: bernstein not imply vitali}
There is a model $\mathcal{M}$ of $ZF$ in which there is a Bernstein set but no Vitali set.
\end{theorem}

In \cite{Wang.Wu.Yu:13}, a model  $\mathcal{M}$ of $ZF$ was constructed so that there is a cofinal chain of Turing degrees of order type $\omega_1$ but there is no well ordering of reals. We claim that the  $\mathcal{M}$ is exactly what we want.

Fix a cofinal chain $\{\mathbf{\mathbf{x}}_{\alpha}\}_{\alpha<\omega_1}$ of Turing degrees order type of $\omega_1$. 

\begin{lemma}
There is no a Vitali set in $\mathcal{M}$.
\end{lemma}
\begin{proof}
Assume otherwise. Then we may pick a unique real from each Turing degree $\mathbf{\mathbf{x}}_{\alpha}$. Since  $\{\mathbf{\mathbf{x}}_{\alpha}\}_{\alpha<\omega_1}$ is cofinal, we have a well ordering of reals in $\mathcal{M}$ which is a contradiction.
\end{proof}

We construct a Bernstein set in $\mathcal{M}$. We define a set $B=\bigcup_{\alpha<\omega_1}B_{\alpha}$.

We shall let $B_{\alpha}$ only contain reals which Turing below $\mathbf{x}_{\alpha}$ for limit ordinal $\alpha$.

Let $B_0=\emptyset$.

At limit stage $\alpha$, just let  $B_{\alpha}=\bigcup_{\gamma<\alpha}B_{\gamma}$.

At stage $\alpha+1$ but $\alpha$ is not limit,  just let  $B_{\alpha+1}=B_{\alpha}$. 

At stage $\alpha+1$ but $\alpha$ is limit. Let $\mathcal{A}_{\alpha}$ be the collection of perfect trees Turing below $\mathbf{x_{\alpha}}$. Let $$B_{\alpha+1}= B_{\alpha}\cup \{z\mid \exists T \in \mathcal{A}_{\alpha}(z\in T\wedge z\not\leq_T \mathbf{x}_{\alpha} \wedge z \leq_T \mathbf{x}_{\alpha+1})\}.$$ Note that for each $T\in \mathcal{A}_{\alpha}$, there is a real $x\in [T]$ so that $x\oplus T\equiv_T \mathbf{x}_{\alpha+1}$. So $B_{\alpha+1}\cap [T]$ is not empty for each $T\in\mathcal{A}_{\alpha}$. 

This finishes the construction.

Since  $\{\mathbf{\mathbf{x}}_{\alpha}\}_{\alpha<\omega_1}$ is cofinal, by the construction, $B\cap [T]$ is not empty for any perfect tree $T$. 

Moreover, for any recursive tree $T$, let $\alpha$ be the least limit ordinal so that $T\leq_T \mathbf{x}_{\alpha}$.  There must be some $z\in [T]$ so that $z\oplus T\equiv_T \mathbf{x}_{\alpha+2}$. Then for such a real $z$, $z\not\in B$. 

So $B$ is a Bernstein set. 

The following question seems unknown.
\begin{question}
Is there a model $\mathcal{M}$ of $ZF$ in which there is a Vitali set but no Bernstein set?
\end{question}

\section{Random linear orders and equivalence relations}
\newcommand{\LO}{\mbox{\rm \textsf{LO}}}
\newcommand{\EQ}{\mbox{\rm \textsf{EQ}}}

Fouch\'e and Nies discussed at the Wollic 2012 meeting in Buenos Aires.  They studied  how to define  randomness for infinite objects other than sets of natural numbers. They didn't want to use any general background such as computable probability spaces (e.g., \cite{Gacs:05}, \cite{Hoyrup.Rojas:09}), but rather restrict attention to a simple kind of relational structures with domain $\NN$, and define randomness for them in a direct way. 

$S_\infty$ denotes the group of permutations of $\NN$. Fix a  Borel class $\+ C \sub \+ P(\NN^k)$ closed under permutations $p \in S_\infty$, such as the  linear orders, or the equivalence relations. We say that a measure $\mu $ on $\+ C$ is \emph{invariant}  if   for each measurable $\+ D \sub \+ C$, 

\[ \fa p \in S_\infty  \, [ \mu  (p (\+ D)) = \mu (\+ D) ].\]

\subsection{Linear orders on $\NN$}
Let $\LO$ denote the class of reflexive linear orders on $\NN$. 
Using methods of topological dynamics, Glasner and Weiss \cite{Glasner.Weiss:02} showed that there is a unique invariant probability measure on $\LO$. The uniqueness is the important part, as it shows that this measure is canonical. 
  
Fouch\'e \cite{} showed that this ``Glasner-Weiss'' measure $\mu_{GW}$  is computable.   For any distinct $a_0, \ldots, a_{n-1} \in \NN$ let 

\[ [a_0, \ldots, a_{n-1}] = \{ L \in \LO \colon \, a_0 <_L \ldots <_L a_{n-1} \}. \]

Clearly, for any invariant measure  $\mu$ on $\LO$, we must have 

\[ \mu  [a_0, \ldots, a_{n-1}]  = 1/n!. \]

Another approach to obtain this measure is as follows. There is a natural correspondence between $\LO$ and  the class of functions $[T]$, where $T$ is the tree  
\[ \{ f\in \NN^\NN \colon \, \fa i  \, f(i) \le i \}. \]
Namely,  $f(i) = k$ means: If $k=0$  put $i$ as a new least, and if $k=i$ put $i$ as a new greatest element.
   Otherwise  we put $i$ between the $k-1$--th and the $k$--th  element of the linear order we already have on $\{0, \ldots, i-1\}$. 
   
   Now let $\mu$ be the product  measure $\Pi_ i \mu_i$ where $\mu_i$ is  the uniform probability measure on $\{0, \ldots, i\}$.
   
   We can now export the known randomness notions to  $[T]$ and thereby to $\LO$. For instance, every Kurtz-random linear order is of type $\QQ$. 
The following analog of  the Levin-Schnorr Theorem   yields a characterization of ML-random linear orders via the prefix-free initial segment complexity.

\begin{proposition}\label{pro:Schnorr for LO}
	Let $f \in [T]$. Then $f$ is ML-random $\LLR$ 
	
	\hfill  $\ex b \fa n \, K(f \uhr n) \ge \log_2 (n!) -b$. 
\end{proposition}

\begin{proof} Adapt the usual proof of Levin-Schnorr (for instance \cite[3.2.9]{Nies:book}). Thus, let  \[\+ U_b= \{f\in [T] \colon \ex   n \, K(f \uhr n) < \log_2 (n!) -b \}\]
and show that $\seq {\+ U_b}\sN b$	is a universal \ML\ test.
\end{proof}
 Melnikov et al.\  \cite{Melnikov.Nies:12} defined $K$-triviality for functions in Baire space. This now yields a notion of $K$-trivial linear orders. Lower c.e.\ functions are the generalizations of c.e.\ sets (what does this mean for linear orders?). We now have everything together and could study  whether in the setting of $\LO$, anything new happens in comparison with the interactions of lowness and randomness on the subsets of $\NN$.

Note that correspondence between $[T]$ and $\LO$ is useful to get the definitions right, but it does not cohere with the  permutation invariance of $\LO$. This suggests that new things will happen  not necessarily  for $[T]$, but    for $\LO$.  
   
\subsection{Equivalence relations on $\NN$}

Here there is an  invariant probability  measures, though it is not  unique. To define  it, pick a probability computable measure $\gamma$  on $\NN$, such as  the one given by \[\gamma (\{i\})= \frac 6 {\pi^2} i^{-2}.\]
Let $\mu^*$ be its power $  \gamma^\NN$ which is a computable measure on Baire space $\NN^\NN$.   Mapping a function  $f$ to its  kernel   yields an onto map $\Phi$  from  $\NN^\NN$ to the set $\EQ$ of equivalence relations on $\NN$. Let $\mu $ on $\EQ$  be the image measure of $\mu^*$. 

Since $\Phi$ respects permutations, clearly $\mu$ is invariant. Another $\gamma$ would yield a different invariant measure. 

We note that if $\gamma (\{i\}) \neq 0$ for each $i$, then a.e.\ [$\mu$] equivalence relation $E$ is of type $(\infty, \infty)$, that is,  has infinitely many classes that are  all infinite. This is so because for any fixed  $n,k$, the class of functions $f$ with $f^{-1}(n)\sub \{0, \ldots, k-1\}$   has measure $0$.

\subsection{Other examples}

Bakh Khoussainov has a draft (8 pages, late 2012) on random algebras in a finite signature. It's not clear whether there is an invariant measure leading to his definition. 

One could also consider random Polish spaces, where the structure is given by a countable dense set with distance relations $R_{< q}xy$ denoting that $d(x,y)< q$ (where $q \in \QQ^+$). 
\newpage


\def\cprime{$'$}

\end{document}